\documentclass[12pt,british]{article}

\usepackage{babel}
\usepackage{mathrsfs,diagbox} 
\usepackage{amsmath,amsfonts,amssymb,amsthm}
\usepackage{ucs} 
\usepackage[utf8x]{inputenc} 
\usepackage[T1]{fontenc} 
\usepackage{enumerate}
\usepackage{textcomp,url}
\usepackage{eucal}
\usepackage{comment}
\usepackage[noBBpl]{mathpazo}
\usepackage[matrix,arrow]{xy}
\usepackage{epsfig}
\usepackage[pdftex,                %
implicit=true,
bookmarks=true,
breaklinks=true,
bookmarksnumbered = true,
colorlinks= true,
urlcolor= green,
anchorcolor = yellow,
citecolor=blue,
pdfborder= {0 0 0}
]{hyperref}

\usepackage{multicol}

\makeatletter

\renewcommand{\p@enumii}{}
\def\@enum@{\list{\csname label\@enumctr\endcsname}%
{\usecounter{\@enumctr}\def\makelabel##1{
\normalfont\ignorespaces\emph{{##1}~}}
\setlength{\labelsep}{3pt}
\setlength{\parsep}{0pt}
\setlength{\itemsep}{0pt}
\setlength{\leftmargin}{0pt}
\setlength{\labelwidth}{0pt}
\setlength{\listparindent}{\parindent}
\setlength{\itemsep}{0pt}
\setlength{\itemindent}{0pt}
\topsep=3pt plus 1pt minus 1 pt}}
\makeatother

\renewcommand{\epsilon}{\ensuremath{\varepsilon}}
\renewcommand{\phi}{\ensuremath{\varphi}}

\renewcommand{\to}{\ensuremath{\longrightarrow}}
\renewcommand{\mapsto}{\ensuremath{\longmapsto}}
\newcommand{\R}{\ensuremath{\mathbb R}}
\newcommand{\N}{\ensuremath{\mathbb N}}
\newcommand{\Z}{\ensuremath{\mathbb Z}}

\newcommand{\sn}[1][n]{\ensuremath{S_{{#1}}}}
\newcommand{\an}[1][n]{\ensuremath{A_{{#1}}}}

\DeclareRobustCommand*{\up}[1]{\textsu{#1}}
\DeclareRobustCommand*{\up}[1]{\textsuperscript{#1}}

\renewcommand{\ker}[1]{\ensuremath{\operatorname{\text{Ker}}\left({#1}\right)}}
\newcommand{\im}[1]{\ensuremath{\operatorname{\text{Im}}\left({#1}\right)}}
\newcommand{\aut}[1]{\ensuremath{\operatorname{\text{Aut}}\left({#1}\right)}}
\newcommand{\aff}[1]{\ensuremath{\operatorname{\text{Aff}}\left({#1}\right)}}
\newcommand{\id}{\ensuremath{\operatorname{\text{Id}}}}

\newcommand{\bernoulli}[1][n]{\ensuremath{\mathcal{B}_{#1}}}

\makeatletter
\def\@map#1#2[#3]{\mbox{$#1 \colon\thinspace #2 \to #3$}}
\def\map#1#2{\@ifnextchar [{\@map{#1}{#2}}{\@map{#1}{#2}[#2]}}
\makeatother

\newcommand{\brak}[1]{\ensuremath{\left\{ #1 \right\}}}
\newcommand{\ang}[1]{\ensuremath{\left\langle #1\right\rangle}}

\newcommand{\setangl}[2]{\ensuremath{\ang{\left. #1 \,\right\rvert \, #2}}}
\newcommand{\ord}[1]{\ensuremath{\left\lvert #1\right\rvert}}

\newcommand{\setl}[2]{\ensuremath{\brak{\left. #1 \,\right\rvert \, #2}}}
\newcommand{\lhra}{\lhook\joinrel\longrightarrow}

\newtheoremstyle{theoremm}{}{}{\itshape}{}{\scshape}{.}{ }{}
\theoremstyle{theoremm}

\newtheorem{thm}{Theorem}
\newtheorem{lem}[thm]{Lemma}
\newtheorem{prop}[thm]{Proposition}
\newtheorem{cor}[thm]{Corollary}

\newtheoremstyle{remark}{}{}{}{}{\scshape}{.}{ }{}
\theoremstyle{remark}

\newtheorem{defn}[thm]{Definition}
\newtheorem{rem}[thm]{Remark}
\newtheorem{rems}[thm]{Remarks}

\newtheoremstyle{comment}{}{}{\bfseries}{}{\bfseries}{:}{ }{}
\theoremstyle{comment}

\newcommand{\redef}[1]{Definition~\protect\ref{def:#1}}
\newcommand{\reth}[1]{Theorem~\protect\ref{th:#1}}
\newcommand{\relem}[1]{Lemma~\protect\ref{lem:#1}}
\newcommand{\repr}[1]{Proposition~\protect\ref{prop:#1}}
\newcommand{\reco}[1]{Corollary~\protect\ref{cor:#1}}
\newcommand{\resec}[1]{Section~\protect\ref{sec:#1}}

\newcommand{\rerem}[1]{Remark~\protect\ref{rem:#1}}
\newcommand{\rerems}[1]{Remarks~\protect\ref{rem:#1}}
\newcommand{\req}[1]{equation~(\protect\ref{eq:#1})}
\newcommand{\reqref}[1]{(\protect\ref{eq:#1})}
\allowdisplaybreaks

\usepackage[x11names]{xcolor}

\newcommand{\comj}[1]{\noindent\textcolor{blue}{\textbf{!!!J!!!~#1}}}

\newcommand{\comdo}[1]{\noindent\textcolor{red}{\textbf{!!!DO!!!~#1}}}

\begin{document}

\title{Almost-crystallographic groups as quotients of Artin braid groups}

\author{DACIBERG~LIMA~GON\c{C}ALVES\\
Departamento de Matem\'atica - IME-USP,\\
Rua~do~Mat\~ao~1010~CEP:~05508-090  - S\~ao Paulo - SP - Brazil.\\
e-mail:~\url{dlgoncal@ime.usp.br}\vspace*{4mm}\\
JOHN~GUASCHI\\
Normandie Univ., UNICAEN, CNRS,\\
Laboratoire de Math\'ematiques Nicolas Oresme UMR CNRS~\textup{6139},\\
CS 14032, 14032 Cedex Cedex 5, France.\\
e-mail:~\url{john.guaschi@unicaen.fr}\vspace*{4mm}\\
OSCAR~OCAMPO~\\
Universidade Federal da Bahia,\\
Departamento de Matem\'atica - IME,\\
Av. Adhemar de Barros~S/N~CEP:~40170-110 - Salvador - BA - Brazil.\\
e-mail:~\url{oscaro@ufba.br}
}

\date{10th November 2017}

\maketitle

\begin{abstract}
\noindent
\emph{Let $n,k\geq 3$. In this paper, we analyse the quotient group $B_n/\Gamma_k(P_n)$ of the Artin braid group $B_{n}$ by the subgroup $\Gamma_k(P_n)$ belonging to the lower central series of the Artin pure braid group $P_{n}$. We prove that it is an almost-crystallographic group. We then focus more specifically on the case $k=3$.
If $n\geq 5$, and $if \tau\in \N$ is such that $\gcd(\tau,6)=1$, we show that $B_n/\Gamma_3(P_n)$ possesses torsion $\tau$ if and only if $\sn$ does,  
and we  prove that there is a one-to-one correspondence between the conjugacy classes of elements of order $\tau$ in $B_n/\Gamma_3(P_n)$ with those of elements of order $\tau$ in the symmetric group $\sn$. We also exhibit a presentation for the almost-crystallographic group $B_n/\Gamma_3(P_n)$.  
Finally, we obtain some $4$-dimensional almost-Bieberbach subgroups of $B_3/\Gamma_3(P_3)$, we explain how to obtain almost-Bieberbach subgroups of $B_4/\Gamma_3(P_4)$ and $B_3/\Gamma_4(P_3)$, and we exhibit explicit elements of order $5$ in $B_5/\Gamma_3(P_5)$.}
 \end{abstract}

\section{Introduction}



In this paper, we continue our study of quotients of the Artin braid group $B_{n}$ by elements of the lower central series
$(\Gamma_k(P_n))_{k\in \N}$ of the Artin pure braid group $P_n$.  In the paper~\cite{GGO}, we analysed the group $B_n/\Gamma_2(P_n)$ in some detail, and we proved notably that it is a crystallographic group. Using different techniques, I.~Marin generalised the results of~\cite{GGO} to generalised braid groups associated to arbitrary complex reflection groups~\cite{Ma}. In the present paper, we show for all $n, k\geq3$, the quotient $B_n/\Gamma_k(P_n)$ of $B_{n}$ by $\Gamma_k(P_n)$ is an almost-crystallographic group, and we investigate more thoroughly the group $B_n/\Gamma_3(P_n)$.
As in~\cite{GGO}, some natural questions that arise are the existence or not of torsion, the realisation of elements of finite order and that of finite subgroups, their conjugacy classes and the relation with other types of group, such as (almost-) crystallographic groups.



This paper is organised as follows. In \resec{prelim}, we recall some definitions and facts about the Artin braid groups, their quotients by the elements of the lower central series $(\Gamma_{k}(P_{n}))_{k\in \N}$ of $P_n$, and almost-crystallographic groups. In \resec{acgroups}, we discuss the quotient $B_n/\Gamma_k(P_n)$, where $n,k\geq 3$. If $G$ is a group, for all $q\in \N$, let $L_q(G)$ denote the lower central series quotient $\Gamma_q(G)/\Gamma_{q+1}(G)$. These quotients have been widely studied, see~\cite{Hal,MKS} for example. In the case of $P_n$, it is known that $L_q(P_n)$ is a free Abelian group of finite rank, and by~\cite[Theorem~4.6]{LVW}, its rank is given by:
\begin{equation}\label{eq:ranklvw}
\operatorname{rank}(L_q(P_n))=\frac{1}{q}\sum_{j=1}^{n-1}\sum_{d\mid q}\,\mu(d)j^{q/d} 
\end{equation}
where $\mu$ is the M\"obius function. From this, it follows that the nilpotent group $P_n/\Gamma_k(P_n)$ of nilpotency class $k-1$ is also torsion free (see \relem{ptorsion}(\ref{it:ptorsiona})). Using~\reqref{ranklvw}, in \repr{nilmanifold}, we calculate the \emph{Hirsch length} of $P_n/\Gamma_k(P_n)$, which is equal to $\sum_{i=1}^{k-1}\operatorname{rank}(L_i(P_n))$. In particular, the Hirsch length of $P_n/\Gamma_3(P_n)$ and $P_n/\Gamma_4(P_n)$ is equal to $\binom{n}{2} + \binom{n}{3}$ and $\binom{n}{2} + \binom{n}{3} + 2 \binom{n+1}{4}$ respectively (see~\cite[Theorem~1.1]{CS}).
Using a criterion given in~\cite{Dekimpe}, we are then able to show that $B_n/\Gamma_k(P_n)$ is an almost-crystallographic group. 

\begin{thm}\label{th:acgroups}
Let $n,k\geq 3$. The group $B_n/\Gamma_k(P_n)$ is an almost-crystallographic group whose holonomy group is $\sn$ and whose dimension is equal to $\sum_{q=1}^{k-1}\left( \frac{1}{q}\sum_{j=1}^{n-1}\sum_{d\mid q}\mu(d)j^{q/d} \right)$. In particular, the dimension of $B_n/\Gamma_3(P_n)$ (resp.\ of $B_n/\Gamma_4(P_n)$) is equal to $\binom{n}{2} + \binom{n}{3}$  (resp.\ to $\binom{n}{2} + \binom{n}{3} + 2 \binom{n+1}{4}$).
\end{thm}

 Torsion-free almost-crystallographic groups, or \emph{almost-Bieberbach} groups, are of particular interest because they arise as fundamental groups of infra-nilmanifolds. Infra-nilmanifolds are manifolds that are finitely covered by a nilmanifold and represent a natural generalisation of flat manifolds. They play an important r\^{o}le in dynamical systems, notably in the  study of expanding maps and Anosov diffeomorphisms~\cite{DD}. The reader may consult~\cite{Dekimpe,Gromov,Ruh} for more information about these topics.


Another interesting problem is that of the nature of the finite-order elements of $B_n/\Gamma_k(P_n)$. Knowledge of the torsion of this group may be used for example to construct almost-Bieberbach subgroups. In this direction, we prove \reth{almostb}, which generalises~\cite[Theorem~2]{GGO}.


\begin{thm}\label{th:almostb}
Let $n,k\geq 3$. Then the quotient group $B_n/\Gamma_k(P_n)$ has no elements of order $2$ nor of order $3$.
\end{thm}

From this, we are able to deduce \reco{almostb}, which proves the existence of almost-Bieberbach groups in $B_n/\Gamma_k(P_n)$ for all $n,k\geq 3$. In much of the rest of the paper, we focus our attention on the case $k=3$. To study the torsion of $B_n/\Gamma_3(P_n)$, an explicit basis of $\Gamma_2(P_n)/\Gamma_3(P_n)$ is introduced in \resec{bngammakpn}, and in \resec{partition}, we partition this basis into the orbits of the action by conjugation of a certain element $\delta_{n}$ of $B_n/\Gamma_3(P_n)$. In \resec{presgamma3}, we exhibit presentations of $P_n/\Gamma_3(P_n)$ and $B_n/\Gamma_3(P_n)$. This enables us to study the finite-order elements of $B_n/\Gamma_3(P_n)$ and their conjugacy classes in \resec{torsion2}. The following result shows that the torsion of $B_n/\Gamma_3(P_n)$ coincides with that of $\sn$ if we remove the elements whose order is divisible by $2$ or $3$.

\begin{thm}\label{th:torsion}
Let $n\geq 5$, and let $\tau\in \N$ be such that $\gcd(\tau,6)=1$. Then the group $B_n/\Gamma_3(P_n)$ admit finite-order elements of torsion $\tau$ if and only if $\sn$ does. Further, if $x\in \sn$ is of order $\tau$, there exists $\alpha\in \sn$ of order $\tau$ such that $\overline{\sigma}(\alpha)=x$, in particular $\overline{\sigma}(\alpha)$ and $x$ have the same cycle type. 
\end{thm}

We end \resec{torsion} with an analysis of the conjugacy classes of the finite-order elements of $B_n/\Gamma_3(P_n)$, the main result in this direction being the following.


\begin{thm}\label{th:conjugacy}
Let $n\geq 5$, and let $\alpha$ and $\beta$ be two finite-order elements of $B_n/\Gamma_3(P_n)$ whose associated permutations have the same cycle type. Then $\alpha$ and $\beta$ are conjugate in $B_n/\Gamma_3(P_n)$.
\end{thm}

In \resec{small}, we discuss some aspects of the quotients $B_n/\Gamma_3(P_n)$, where $n$ is small. In \resec{gamma3}, we obtain some almost-Bieberbach subgroups of $B_3/\Gamma_3(P_3)$ of dimension $4$ that are the fundamental groups of orientable $4$-dimensional infra-nilmanifolds, and in 
\resec{gamma5}, we compute $\delta_{5}^{5}$ in terms of the chosen basis of $\Gamma_2(P_5)/\Gamma_3(P_5)$. Using the constructions of \resec{torsion2}, this allows us to exhibit explicit elements of order $5$ in $B_5/\Gamma_3(P_5)$.

Another important question in our study is the existence and embedding of finite groups in $B_n/\Gamma_3(P_n)$. For cyclic groups, the answer is given by \reth{torsion}, 
and this may be generalised to Abelian groups in \reco{existAb} using \relem{1to1}(\ref{it:1to1b}). 
We have proved recently in~\cite{GGO1} that Cayley-type theorems hold for $B_n/\Gamma_2(P_n)$ and $B_n/\Gamma_3(P_n)$, namely that if $G$ is any finite group of odd order $n$ (resp.\ of order $n$ relatively prime with $6$) then $G$ embeds in $B_n/\Gamma_2(P_n)$ (resp.\ in $B_n/\Gamma_3(P_n)$). In the case of $B_n/\Gamma_2(P_n)$, the same result has been proved independently by V.~Beck and I.~Marin within the more general setting of complex reflection groups~\cite{MaV}. In the same paper, we also show that with appropriate conditions on $n$ and $m$, two families of groups of the form  $G=\Z_n\rtimes_{\theta} \Z_m$ embed in $B_n/\Gamma_2(P_n)$ and $B_n/\Gamma_3(P_n)$.



\subsection*{Acknowledgements}

Work on this paper started in 2014, and took place during several visits to the Departamento de Matem\'atica, Universidade Federal de Bahia, the Departamento de Matem\'atica do IME~--~Universidade de S\~ao Paulo and to the Laboratoire de Math\'ematiques Nicolas Oresme UMR CNRS~6139, Universit\'e de Caen Normandie. The first two authors were partially supported by  the FAPESP Projeto Tem\'atico Topologia Alg\'ebrica, Geom\'etrica e Diferencial 2012/24454-8 (Brazil), by the R\'eseau Franco-Br\'esilien en Math\'ematiques, by the CNRS/FAPESP project n\up{o}~226555 (France) and n\up{o}~2014/50131-7 (Brazil), and the CNRS/FAPESP PRC project n\up{o}~275209 (France) and n\up{o}~2016/50354-1 (Brazil). The third author was initially supported by 
a project grant n\up{o}~151161/2013-5 from CNPq and then partially supported by the FAPESP Projeto Tem\'atico Topologia Alg\'ebrica, Geom\'etrica e Diferencial n\up{o}~2012/24454-8 (Brazil).

\section{Preliminaries}\label{sec:prelim}

In this section, we recall some definitions and results about Artin braid groups, \mbox{(almost-)} crystallographic groups, and the relations between them that will be used in this paper.

\subsection{Artin braid groups}\label{sec:artin}

We start by recalling some facts about the Artin braid group $B_{n}$ on $n$ strings (see~\cite{Ha} for more details). It is well known that $B_{n}$ possesses a presentation with generators $\sigma_{1},\ldots,\sigma_{n-1}$ that are subject to the following relations:
\begin{equation}\label{eq:artin1}
\left\{ \begin{gathered}
\text{$\sigma_{i} \sigma_{j} = \sigma_{j}  \sigma_{i}$ for all $1\leq i<j\leq n-1$ such that $\ord{i-j}\geq 2$}\\
\text{$\sigma_{i+1} \sigma_{i} \sigma_{i+1} =\sigma_{i} \sigma_{i+1} \sigma_{i}$ for all $1\leq i\leq n-2$.}
\end{gathered}\right.
\end{equation}
Let $\map{\sigma}{B_{n}}[\sn]$ be the homomorphism defined on the given generators of $B_{n}$ by $\sigma(\sigma_{i})=(i,i+1)$ for all $1\leq i\leq n-1$. Just as for braids, we read permutations from left to right so that if $\alpha, \beta \in \sn$, their product is defined by $\alpha\cdot \beta (i)=\beta(\alpha(i))$ for $i=1,2,\ldots, n$. The pure braid group $P_{n}$ on $n$ strings is defined to be the kernel of $\sigma$, from which we obtain the following short exact sequence:
\begin{equation}\label{eq:sespn}
1 \to P_n \to  B_n \stackrel{\sigma}{\to} \sn \to 1.
\end{equation}

Let $G$ be a group. If $g,h\in G$ then $[g,h]=ghg^{-1}h^{-1}$ will denote their commutator, and if $H,K$ are subgroups of $H$ then we set $[H,K]=\setangl{[h,k]}{k\in H,\, k\in K}$. The \textit{lower central series} $\brak{\Gamma_i(G)}_{i\in \N}$ of $G$ is defined inductively by $\Gamma_1(G)=G$, and $\Gamma_{i+1}(G)=[G,\Gamma_i(G)]$ for all  $i\in \N$. If $i=2$, $\Gamma_2(G)$ is the \textit{commutator subgroup} of $G$. For all $i,j\in \N$ with $j>i$, $\Gamma_j(G)$ is a normal subgroup of $\Gamma_i(G)$.
Following P.~Hall, for any group-theoretic property $\mathcal{P}$, $G$ is said to be \textit{residually} $\mathcal{P}$ if for any (non-trivial) element $x\in G$, there exist a group $H$ that possesses property $\mathcal{P}$ and a surjective homomorphism $\phi\colon G\to H$ 
such that $\phi(x)\neq 1$. It is well known that a group $G$ is \textit{residually nilpotent}  if and only if $\cap_{i\geq 1}\Gamma_i(G)=\brak{1}$. 
The lower central series of groups and their sucessive quotients $\Gamma_i(G)/\Gamma_{i+1}(G)$ are isomorphism invariants, and have been widely studied using commutator calculus, in particular for free groups of finite rank~\cite{Hal, MKS}. Falk and Randell, and independently Kohno, investigated the lower central series of the pure braid group $P_n$, and proved that $P_n$ is residually nilpotent~\cite{FR, Ko}. 

A presentation of $P_n$ is given by the set of generators $\brak{A_{i,j}}_{1\leq i<j\leq n}$, where:
\begin{equation}\label{eq:defaij}
A_{i,j}=\sigma_{j-1}\cdots \sigma_{i+1}\sigma_{i}^{2} \sigma_{i+1}^{-1}\cdots \sigma_{j-1}^{-1},
\end{equation}
subject to the following relations that are expressed in terms of commutators (see \cite[Remark~3.1,~p.~56]{MK} or~\cite[Chapter~1,~Lemma~4.2]{Ha}):
\begin{equation}\label{eq:purerelations}
\begin{cases}
[A_{r,s},A_{i,j}]=1 & \text{if $1\leq r<s<i<j\leq n$ or $1\leq r<i<j<s\leq n$}\\
[A_{r,s},A_{r,j}]=[A_{s,j}^{-1},A_{r,j}] & \text{if $1\leq r<s<j\leq n$}\\
[A_{r,s},A_{s,j}]=[A_{s,j}^{-1},A_{r,j}^{-1}] & \text{if $1\leq r<s<j\leq n$}\\
[A_{r,i},A_{s,j}]=\left[[A_{i,j}^{-1},A_{r,j}^{-1}],A_{s,j}\right] & \text{if $1\leq r<s<i<j\leq n$}.
\end{cases} 
\end{equation}
For notational reasons, if $1\leq i<j\leq n$, we set $A_{j,i}=A_{i,j}$, and if $A_{i,j}$ appears in a word of $P_{n}$ with exponent $m_{i,j}\in \Z$, then we let $m_{j,i}=m_{i,j}$. It follows from the presentation~\reqref{purerelations} that $P_n/\Gamma_2(P_n)$ is isomorphic to $\Z^{n(n-1)/2}$, and that a basis of $P_n/\Gamma_2(P_n)$ is given by $\brak{A_{i,j}}_{1\leq i<j\leq n}$, where by abuse of notation, the $\Gamma_2(P_n)$-coset of $A_{i,j}$ will also be denoted by $A_{i,j}$. For all $k\geq2$, \req{sespn} gives rise to the following short exact sequence:
\begin{equation}\label{eq:sespnquot}
1 \to  P_n/\Gamma_k(P_n) \to  B_n/\Gamma_k(P_n) \stackrel{\overline{\sigma}}{\to} \sn \to 1,
\end{equation}
where $\map{\overline{\sigma}}{B_n/\Gamma_k(P_n)}[\sn]$ is the homomorphism induced by $\sigma$. 
In much of what follows, we shall be interested in the action by conjugation of $B_n/\Gamma_2(P_n)$ on $\Gamma_2(P_n)/\Gamma_3(P_n)$. 
For all $1\leq k\leq n-1$ and $1\leq i<j\leq n$, the action of $B_n$ on $P_n$ described in \cite[equation~(7)]{GGO} may be rewritten as:
\begin{equation}\label{eq:conjugAij1}
\sigma_k A_{i,j}\sigma_k^{-1}=
\begin{cases}
A_{i,j}^{-1}A_{i,j-1}A_{i,j} & \text{if $j=k+1$ and $i<k$}\\
A_{i,j}^{-1}A_{i-1,j}A_{i,j} & \text{if  $i=k+1$}\\
A_{\sigma_k^{-1}(i),\; \sigma_k^{-1}(j)} & \text{otherwise,}
\end{cases}
\end{equation}
where by abuse of notation, we write $\sigma_k^{-1}(i)=\sigma(\sigma_k^{-1})(i)$, $\sigma$ being as in \req{sespn}. This action was used in~\cite{GGO} to prove the following proposition.

\begin{prop}[{\cite[Proposition~12]{GGO}}]\label{prop:prop12}
Let $\alpha\in B_n/\Gamma_2(P_n)$, and let $\pi$ be the permutation induced by $\alpha^{-1}$. Then for all $1\leq i<j\leq n$, $\alpha A_{i,j}\alpha^{-1}=A_{\pi(i), \pi(j)}$ in $P_n/\Gamma_2(P_n)$.
\end{prop}


For all $\alpha\in B_n/\Gamma_3(P_n)$,  and all $1\leq i<j\leq n$ and $1\leq r<s\leq n$, we claim that:
\begin{equation}\label{eq:conjugalpha}
\alpha [A_{i,j},A_{r,s}]\alpha^{-1}=[\alpha A_{i,j} \alpha^{-1}, \alpha A_{r,s}\alpha^{-1}]= [A_{\pi(i), \pi(j)},A_{\pi(r), \pi(s)}] \,\text{in $P_n/\Gamma_3(P_n)$,}
\end{equation}
where $\pi$ is the permutation induced by $\alpha^{-1}$, and by abuse of notation, $A_{i,j}$ and $A_{r,s}$ are considered as elements of $P_{n}/\Gamma_3(P_n)$. To see this, first note that the first equality of~\reqref{conjugalpha} clearly holds, and that for all $k\in \N$, we have the following short exact sequence:
\begin{equation}\label{eq:gamma23P}
1\to \Gamma_k(P_n)/\Gamma_{k+1}(P_n) \to P_n/\Gamma_{k+1}(P_n) \to P_n/\Gamma_k(P_n)\to 1.
\end{equation}
Taking $k=2$ in~\reqref{gamma23P} and using \repr{prop12}, there exist $\gamma_{1},\gamma_{2}\in \Gamma_2(P_n)/\Gamma_{3}(P_n)$ such that $\alpha A_{i,j} \alpha^{-1}=\gamma_{1} A_{\pi(i), \pi(j)}$ and $\alpha A_{r,s} \alpha^{-1}=\gamma_{2} A_{\pi(r), \pi(s)}$. The second equality of~\reqref{conjugalpha} then follows using standard commutator relations.

\subsection{Almost-crystallographic groups}\label{sec:acg}

In this section, we recall briefly the definitions of almost-crystallographic and almost-Bieberbach groups, which are natural generalisations of crystallographic and Bieber\-bach groups, as well as a characterisation of almost-crystallographic groups.  For more details about crystallographic groups, see~\cite[Section~I.1.1]{Charlap},~\cite[Section~2.1]{Dekimpe} or~\cite[Chapter~3]{Wolf}.

Given a connected and simply-connected nilpotent Lie group $N$, the group $\aff{N}$ of affine transformations of $N$ is defined by $\aff{N}=N\rtimes \aut{N}$, and acts on $N$ by:
\begin{equation*}
\text{$(n,\phi)\cdot m = n\phi(m)$ for all $m,n\in N$ and $\phi\in \aut{N}$}.
\end{equation*}
\begin{defn}[{\cite[Sec.~2.2, p.~15]{Dekimpe}}]\label{def:almostc}
Let $N$ be a connected, simply-connected nilpotent Lie group, and consider a maximal compact subgroup $C$ of $\aut{N}$. A uniform discrete subgroup $E$ of $N\rtimes C$ is called an \emph{almost-crystallographic group}, and its dimension is defined to be that of $N$. A torsion-free, almost-crystallographic group is called an \emph{almost-Bieberbach group}, and the quotient space $E\backslash N$ is called an \emph{infra-nilmanifold}. If further $E\subseteq N$, we say that the space $E\backslash N$ is a \emph{nilmanifold}. 
\end{defn}
It is well known that infra-nilmanifolds are classified by their fundamental group that is almost-crystallographic~\cite{Au}
Every almost-crystallographic subgroup $E$ of the group $\aff{N}$ fits into an extension:
\begin{equation}\label{eq:ses}
 1\to \Lambda \to E \to F \to 1,
\end{equation}
where $\Lambda=E\cap N$ is a uniform lattice in $N$, and $F$ is a finite subgroup of $C$ known as the \emph{holonomy group} of the corresponding infra-nilmanifold $E\backslash N$~\cite{Au}. 
Let $M$ be an infra-nilmanifold whose fundamental group $E$ is almost-crystallographic. Following~\cite[Page~788]{GPS},  we recall the construction of a faithful linear representation associated with the extension~\reqref{ses}. Suppose that the nilpotent lattice $\Lambda$ is of class $c+1$ \emph{i.e.}\ $\Gamma_c(\Lambda)\neq 1$ and $\Gamma_{c+1}(\Lambda)= 1$. 
For $i=1,\ldots,c$, let $Z_i=\Gamma_i(\Lambda)\bigl/\Gamma_{i+1}(\Lambda)$ denote the factor groups of the lower central series $\brak{\Gamma_{i}(\Lambda)}_{i=1}^{c+1}$ of $\Lambda$.
We will assume further that these quotients are torsion free, since this will be the case for the groups that we will study in the following sections. Thus $Z_i\cong \Z^{k_i}$ for all $1\leq i \leq c$ and for some $k_i>0$. The \emph{rank} or \emph{Hirsch number} of $\Lambda$ is equal to $\sum_{i=1}^{c}k_i$.
The action by conjugation of $E$ on $\Lambda$ induces an action of $E$ on $Z_{i}$ which factors through an action of the group $E/\Lambda$ (the holonomy group $F$), because $\Lambda$ acts trivially on $Z_{i}$.
This gives rise to a faithful representation $\theta_F\colon F\to \operatorname{\text{GL}}(n,\Z)$ via the composition:
\begin{equation}\label{eq:theta}
 \theta_F\colon F\to \operatorname{\text{GL}}(k_1,\Z)\times \cdots \times \operatorname{\text{GL}}(k_c,\Z) \to \operatorname{\text{GL}}(n,\Z),
\end{equation}
where $n$, which is the rank of $\Gamma$, is also equal to the dimension of $N$. 
Using~\cite[Remark~2.5]{GPS}, this representation will be used in \resec{gamma3} to decide whether $M$ is orientable or not. 
In order to prove \reth{acgroups}, we shall use part of the algebraic characterisation of almost-crystallographic groups given in~\cite[Theorem~3.1.3]{Dekimpe} as follows. 

\begin{thm}[{\cite[Theorem~3.1.3]{Dekimpe}}]\label{th:dekimpe}
Let $E$ be a polycyclic-by-finite group. Then $E$ is almost-crystallographic if and only if it has a  nilpotent subgroup, and possesses no non-trivial finite normal subgroups.
\end{thm}

\section{The almost-crystallographic group $B_n/\Gamma_k(P_n)$}\label{sec:acgroups}


Let $n,k\geq3$. 
In this section, we study the group $B_n/\Gamma_k(P_n)$. In \resec{rankGamma}, we start by recalling some results about the quotient groups $\Gamma_k(P_n)/\Gamma_{k+1}(P_n)$ that appear in~\reqref{gamma23P}. In \resec{bngammakpn}, we prove Theorems~\ref{th:acgroups} and~\ref{th:almostb} which state that the groups $B_n/\Gamma_k(P_n)$ are almost-crystallographic, and that $B_n/\Gamma_k(P_n)$ possesses no element of order $2$ or $3$ respectively. This allows us to prove in \reco{almostb}, which shows that if $H$ is a subgroup of $\sn$ whose order is not divisible by any prime other than $2$ or $3$ then the subgroup $\overline{\sigma}^{-1}(H)/\Gamma_k(P_n)$ of $B_{n}/\Gamma_k(P_n)$ is almost-Bieberbach.

\subsection{The rank of the free Abelian group $\Gamma_k(P_n)/\Gamma_{k+1}(P_n)$}\label{sec:rankGamma} 



Let $n\geq 2$ and $k\geq 1$. 
The group $\Gamma_k(P_n)/\Gamma_{k+1}(P_n)$, which we shall denote by $L_k(P_n)$, is free Abelian of finite rank, and by~\cite[Theorem~4.2]{FRinv} and~\cite[Theorem~4.5]{Ko}, its rank is related to the Poincar\'e polynomial of certain hyperplane complements.
Using Chen groups, Cohen and Suciu gave explicit formul\ae\ for $\operatorname{rank}(L_k(P_n))$ for $k\in \brak{1,2,3}$ and all $n\geq 2$~\cite[Theorem~1.1 and page~46]{CS}.  More generally, by~\cite[Theorem~4.6]{LVW}, the rank of $L_k(P_n)$ is given by \req{ranklvw}. 
In practice, we may compute these 
ranks as follows. If $k\geq 2$, let $k^{\ast}$ be the product of the distinct prime divisors of $k$. 
Then~\reqref{ranklvw} may be rewritten as: 
\begin{equation}\label{eq:reduction}
\operatorname{rank}(L_k(P_n))=\frac{1}{k}\sum_{j=1}^{n-1}\sum_{d|k^\ast}\mu(d)j^{k/d} =\frac{1}{k} \sum_{d|k^\ast}\mu(d)S_{k/d}(n),
\end{equation}
where $S_r(n)=\sum_{j=1}^{n-1}j^r$.
The number of summands in the expression $\frac{1}{k}\sum_{d|k^\ast}\mu(d)j^{k/d}$ is equal to $2^{t}$.

\begin{prop}
Let $n\geq 2$. Then $\operatorname{rank}(L_k(P_n))$
is a polynomial of degree $k+1$ in the variable $n$.
\end{prop}

\begin{proof}
By~\cite[Chapter~XVI]{N}, the sum $S_{k/d}(n)$ is a polynomial of degree $\frac{k}{d}+1$ in the variable $n$, so $\operatorname{rank}(L_k(P_n))$ is also a polynomial in the variable $n$. The result follows by noting that $d=1$ divides $k^\ast$.
\end{proof}

\begin{rems}\label{rem:rank1}\mbox{}
\begin{enumerate}[(a)]
\item For small values of $r$, a polynomial expression for $S_r(n)$ was computed for example in~\cite[Section~1.2]{AIK} for $1\leq r\leq 6$, and~\cite[Chapter~XVI, page~296]{N} or \cite[Tables~I~and~II]{Wi} for $1\leq r\leq 10$.

\item\label{it:rank1b} Using 
the description given in \req{reduction},
we have computed $\operatorname{rank}(L_k(P_n))$ for all 
$1\leq k\leq 10$. 
If $n=2$ or $3$, we obtain the equalities $\operatorname{rank}(L_2(P_n))=\binom{n}{3}$ and $\operatorname{rank}(L_3(P_n))=2\binom{n+1}{4}$ given in~\cite[Theorem~1.1]{CS}.
For $3\leq k\leq 10$, $\operatorname{rank}(L_k(P_n))$ is the product 
of $\binom{n+1}{4}$ by a polynomial in the variable $n$ of degree $k-3$.
The authors do not know whether this is true in general.

\item The numbers $\operatorname{rank}(L_k(P_n))$
may be expressed in terms of Bernoulli numbers since the latter are closely related to sums of powers of consecutive integers (see~\cite{AIK} for more information about Bernoulli numbers, and especially Formula~(1.1) on page~1).
\end{enumerate}
\end{rems}

\subsection{The quotient groups $B_n/\Gamma_k(P_n)$}\label{sec:bngammakpn}

The proofs of the first two results of this section are straightforward, but will be useful in the analysis of the group $B_n/\Gamma_k(P_n)$. First note that $P_n/\Gamma_k(P_n)$ is a nilpotent group of nilpotency class $k-1$, and as we shall see in \relem{ptorsion}(\ref{it:ptorsiona}), it is torsion free.

\begin{prop}\label{prop:nilmanifold}
Let $n,k\geq 3$. The Hirsch length of the nilpotent group $P_n/\Gamma_k(P_n)$
is equal to $\sum_{q=1}^{k-1}\left( \frac{1}{q}\sum_{j=1}^{n-1}\sum_{d|q^\ast}\mu(d)j^{q/d} \right)$. 
In particular, the Hirsch length of $P_n/\Gamma_3(P_n)$ (resp.\ of $P_n/\Gamma_4(P_n)$) is
equal to $\binom{n}{2} + \binom{n}{3}$ (resp.\ to $\binom{n}{2} + \binom{n}{3} + 2 \binom{n+1}{4}$). 
\end{prop}

\begin{proof}
Let $n,k\geq 3$. Since the Hirsch length of a nilpotent group is equal to the sum of the ranks of the consecutive lower central series quotients, 
the first part of the statement follows from \req{reduction}.
If $q\in \brak{3,4}$, the formul\ae\ is then a consequence of those given in~\rerems{rank1}(\ref{it:rank1b}).
 \end{proof}

\begin{lem}\mbox{}\label{lem:ptorsion}
\begin{enumerate}[(a)]
\item\label{it:ptorsiona} Let
$n,k\geq 2$. Then the group $P_n/{\Gamma_k(P_n)}$ is torsion free.

\item\label{it:ptorsionb} Let $n\geq 3$, let $k\geq l\geq 1$, and let $G$ be a finite group. If $B_n/\Gamma_k(P_n)$ possesses a (normal) subgroup isomorphic to $G$ then $B_n/\Gamma_l(P_n)$ possesses a (normal) subgroup isomorphic to $G$. In particular, if $p$ is prime, and if $B_n/\Gamma_l(P_n)$ has no $p$-torsion then $B_n/\Gamma_k(P_n)$ has no $p$-torsion.
\end{enumerate}
\end{lem}

\begin{proof}\mbox{}
\begin{enumerate}
\item The proof is by induction on $k$. If $k=2$ then $P_n/\Gamma_2(P_n)\cong \Z^{n(n-1)/2}$, which implies the result in this case. Suppose then that 
the result holds for some $k\geq 2$. Then the quotient of the short exact sequence~\reqref{gamma23P} is torsion free by induction, and the kernel is torsion free by the results mentioned at the beginning of \resec{rankGamma}. It follows that $P_n/{\Gamma_{k+1}(P_n)}$ is also torsion free.

\item Assume that $B_n/\Gamma_k(P_n)$ possesses a (normal) subgroup isomorphic to $G$. For all $j\geq 2$, we have a central extension of the form:
\begin{equation}\label{eq:gammaquot}
1 \to \Gamma_{j-1}(P_n)/\Gamma_j(P_n) \to B_{n}/\Gamma_j(P_n) \to B_{n}/\Gamma_{j-1}(P_n) \to 1.
\end{equation}
As we mentioned above, the kernel of the short exact sequence~\reqref{gammaquot} is torsion free,
so the restriction of the homomorphism $B_{n}/\Gamma_j(P_n) \to B_{n}/\Gamma_{j-1}(P_n)$ to this subgroup is injective, and $B_{n}/\Gamma_{j-1}(P_n)$ also has a (normal) subgroup isomorphic to $G$. The first part of the statement then follows by reverse induction on $j$. In particular, if $B_n/\Gamma_k(P_n)$ has $p$-torsion, the second part of the statement follows by taking $G=\Z_{p}$.\qedhere
\end{enumerate}
\end{proof}

We now prove \reth{acgroups}.

\begin{proof}[Proof of \reth{acgroups}]
Let $n,k\geq 3$. 
In order to prove that the group $B_n/\Gamma_k(P_n)$ is almost-crystallographic, by \reth{dekimpe}, it suffices to show that it does not have finite (non-trivial) normal subgroups.
Suppose on the contrary that it possesses such a subgroup $H$.
Taking $l=2$ (resp.~$l=1$) in \relem{ptorsion}(\ref{it:ptorsionb}), it follows that $B_n/\Gamma_2(P_n)$ (resp.~$\sn$) has a normal subgroup $H'$ (resp.~$\widehat{H}$) isomorphic to $H$. Hence the groups $\widehat{H}, H'$ and $H$
are either isomorphic to $\sn$, or to the alternating subgroup $\an$, or if $n=4$, to $\Z_{2}\oplus \Z_{2}$. If $n\geq 4$, or if $n=3$ and $\widehat{H}=\sn[3]$, it follows that $H'$ possesses $2$-torsion, but this contradicts~\cite[Theorem~2]{GGO}. It remains to analyse the case $n=3$ and $\widehat{H}=\an[3]\cong \Z_{3}$. Then $H'$ has $3$-torsion, and since the elements of $B_3/\Gamma_2(P_3)$ of order $3$ are pairwise conjugate~\cite[Theorem~5]{GGO}, and $H'$ is normal in $B_3/\Gamma_2(P_3)$, the only elements of $B_3/\Gamma_2(P_3)$ of order $3$ are those belonging to $H'$, but this contradicts the fact that $B_3/\Gamma_2(P_3)$ possesses infinitely many elements of order $3$~\cite[Proposition~21]{GGO}. This proves the first part of the statement of \reth{acgroups}.
For the second part, the dimension of the almost-crystallographic group $B_n/\Gamma_k(P_n)$ is equal to the Hirsch length of $P_n/\Gamma_k(P_n)$ given in \repr{nilmanifold}, and the associated holonomy group is the quotient group $B_n/\Gamma_k(P_n)/P_n/\Gamma_k(P_n)$, which is isomorphic to $\sn$.
\end{proof}

\begin{rems}\mbox{}\label{rem:acgen}
\begin{enumerate}
\item\label{it:acgena}  As in~\cite[Corollary~13]{GGO}, the first part of \reth{acgroups} may be generalised as follows: let $n,k\geq 3$, and let $H$ be a subgroup of $\sn$. Then the group $\sigma^{-1}(H)/\Gamma_{k}(P_{n})$ is an almost almost-crystallographic group whose holonomy group is $H$ and whose dimension is equal to that of $B_{n}/\Gamma_{k}(P_{n})$.
\item 
Using \reth{acgroups}, we may compute the dimension of the almost-crystallographic group $B_n/\Gamma_k(P_n)$. In Table~\ref{table:table1}, we exhibit these dimensions for small values of $n$ and $k$. This gives us an idea about their growth in terms of $k$ and $n$.
\begin{table}[h!]
\centering
\begin{tabular}{ |c||c|c|c|c| } 
 \hline
 \diagbox{$k$}{$n$}
 & $\mathbf{3}$ & $\mathbf{4}$ & $\mathbf{5}$ & $\mathbf{6}$\\
 \hline\hline
$\mathbf{2}$ & 3 & 6 & 10 & 15 \\ 
 \hline
$\mathbf{3}$ & 4 & 10 & 20 & 35 \\ 
 \hline
$\mathbf{4}$ & 6 & 20 & 50 & 105 \\
 \hline
$\mathbf{5}$ & 9 & 41 & 131 & 336 \\
 \hline
\end{tabular}
\caption{The dimension of the almost-crystallographic group $B_n/\Gamma_k(P_n)$ for $n=3,4,5,6$ and $k=2,3,4,5$.}
\label{table:table1}
\end{table}
\end{enumerate}
\end{rems}




%
%

In order to prove \reth{almostb}, we first recall the following lemma from~\cite{GGO}.

\begin{lem}[{\cite[Lemma~28]{GGO}}]\label{lem:lem28}
Let $k,n\geq 3$ and $r\geq 0$ such that $k$ is odd and $r+k\leq n$, and define $\delta_{r,k},\alpha_{r,k}\in B_n/\Gamma_2(P_n)$ by:
\begin{equation}\label{eq:eq14}
\text{$\delta_{r,k}=\sigma_{r+k-1}\cdots \sigma_{r+\frac{k+1}{2}} \sigma_{r+\frac{k-1}{2}}^{-1}\cdots \sigma_{r+1}^{-1}$ and $\alpha_{r,k}=\sigma_{r+1}\cdots \sigma_{r+k-1}$.}
\end{equation}
Then $\delta_{r,k}$ is of order $k$ in $B_n/\Gamma_2(P_n)$, and satisfies:
\begin{equation}\label{eq:eq15}
\delta_{r,k}=\bigl(A_{r+\frac{k+1}{2}, r+k} A_{r+\frac{k+3}{2}, r+k} \cdots A_{r+k-1, r+k}\bigr) \alpha_{r,k}^{-1}.
\end{equation}
\end{lem}

Let $n\geq 3$. We now describe a basis of the free Abelian group $\Gamma_2(P_n)/\Gamma_3(P_n)$ of rank $\binom{n}{3}$ by \rerem{rank1}(\ref{it:rank1b}). 
This group is generated by the elements of the form $[A_{i,j},A_{r,s}]$, where $1\leq i<j\leq n$ and $1\leq r<s\leq n$, but by \req{purerelations}, $[A_{i,j},A_{r,s}]=1$ in $\Gamma_2(P_n)/\Gamma_3(P_n)$ if $\operatorname{\text{Card}}\brak{i,j,r,s}\in \brak{2,4}$, so $\Gamma_2(P_n)/\Gamma_3(P_n)$ is generated by the elements of the form $[A_{i,j},A_{r,s}]$, where $\operatorname{\text{Card}}\brak{i,j,r,s}=3$. Let $\alpha_{i,j,k}=[A_{i,j}, A_{j,k}]$ for all $1\leq i<j<k\leq n$. Using~\reqref{purerelations} once more, the following equalities hold in $\Gamma_2(P_n)/\Gamma_3(P_n)$: 
\begin{equation}\label{eq:alphaijk}
\text{$[A_{i,k}, A_{j,k}]=\alpha_{i,j,k}^{-1}$, $[A_{i,j}, A_{i,k}]=\alpha_{i,j,k}^{-1}$ and $[A_{j,k}, A_{i,j}]=\alpha_{i,j,k}^{-1}$,}
\end{equation}
and if $\alpha_{i,j,k}=\alpha_{r,s,t}^{\pm 1}$ then $\{i,j,k\}=\{r,s,t\}$. So the set
\begin{equation}\label{eq:basis1}
\mathcal{B}=\setl{\alpha_{i,j,k}}{1\leq i<j<k\leq n},
\end{equation}
generates $\Gamma_2(P_n)/\Gamma_3(P_n)$. Since $\operatorname{\text{Card}}(\mathcal{B})=\binom{n}{3}=\operatorname{\text{rank}}(\Gamma_2(P_n)/\Gamma_3(P_n))$, it follows that $\mathcal{B}$ is a basis for $\Gamma_2(P_n)/\Gamma_3(P_n)$, in particular, $\alpha_{i,j,k}\neq 1$.
Applying~\req{conjugalpha} to the Artin generators of $B_{n}$, and using \req{alphaijk}, the action of $B_{n}/\Gamma_{3}(P_{n})$ on the elements of $\mathcal{B}$ given in~\reqref{basis1} is as follows:
\begin{equation}\label{eq:action1}
\sigma_k \alpha_{r,s,t}\sigma_k^{-1}= 
\begin{cases}
\alpha_{\sigma_k(r), \sigma_k(s), \sigma_k(t)} & \text{if $\sigma_k(r)< \sigma_k(s)< \sigma_k(t)$}\\
\alpha_{\sigma_k(s), \sigma_k(r), \sigma_k(t)}^{-1} & \text{if   $\sigma_k(r)> \sigma_k(s)$}\\
\alpha_{\sigma_k(r), \sigma_k(t), \sigma_k(s)}^{-1} & \text{if $\sigma_k(s)> \sigma_k(t)$}
\end{cases}
\end{equation}
for all $1\leq k\leq n-1$ and $1\leq r<s<t\leq n$.
Using the action of $B_n$ on $P_n$ given in \req{conjugAij1} and \req{alphaijk}, one may check that in $B_n/\Gamma_3(P_n)$: 
\begin{equation}\label{eq:action2}
\sigma_k A_{i,j}\sigma_k^{-1}= \begin{cases}
A_{i,j-1}\alpha_{i,j-1,j}^{-1} & \text{if  $j=k+1$ and $i<k$}\\
A_{i-1,j}\alpha_{i-1,i,j}^{-1} & \text{if $i=k+1$}\\
A_{\sigma_k(i), \sigma_k(j)} & \text{otherwise.}
\end{cases}
\end{equation}

\begin{rem}\label{rem:actinj}
Consider the action by conjugation of $B_{n}/\Gamma_{3}(P_{n})$ on $\Gamma_2(P_n)/\Gamma_3(P_n)$ described by~\reqref{action1}. 
The restriction of this action to $P_{n}/\Gamma_{3}(P_{n})$ on $\Gamma_2(P_n)/\Gamma_3(P_n)$  is trivial, and so the action of $B_{n}/\Gamma_{3}(P_{n})$
factors through $\sn$. Further, if $n>3$, the action of $\sn$ is injective, \emph{i.e.}\ if $\alpha$ is a non-trivial permutation, then the induced automorphism is different from the identity by~\reqref{conjugalpha},~\reqref{alphaijk} and the fact that $\mathcal{B}$ is a basis.


\end{rem}

\begin{proof}[Proof of \reth{almostb}]
Let $n,k\geq 3$. By~\cite[Theorem~2]{GGO}, $B_n/\Gamma_2(P_n)$ has no $2$-torsion, so applying \relem{ptorsion}(\ref{it:ptorsionb}), we conclude that this is also the case for $B_n/\Gamma_k(P_n)$. 
To complete the proof, using \relem{ptorsion}(\ref{it:ptorsionb}) once more, it suffices to prove that $B_n/\Gamma_3(P_n)$ has no $3$-torsion. Suppose on the contrary that $B_n/\Gamma_3(P_n)$ possesses an element $\rho$ of order $3$. 
Let $m$ denote the number of conjugacy classes of elements of $B_n/\Gamma_2(P_n)$ of order $3$. By~\cite[Theorem~5]{GGO}, $m$ is equal to the number of conjugacy classes of elements of order $3$ in $\sn[n]$, so $m=\lfloor n/3\rfloor\geq 1$, and using~\cite[Proposition~29]{GGO}, representatives of the conjugacy classes of elements of order $3$ in $B_n/\Gamma_2(P_n)$ are given by the elements of the form $\eta_{t}=\delta_{3t,3}\cdots\delta_{6,3}\delta_{3,3}\delta_{0,3}$, where $0\leq t \leq m-1$. Conjugating $\rho$ if necessary and using the short exact sequence~\reqref{gammaquot}, there exists $\theta\in \Gamma_2(P_n)/\Gamma_3(P_n)$ such that $\rho=\theta\eta_{t}$ for some $0\leq t \leq m-1$ (here $\eta_{t}$ is considered as an element of $B_n/\Gamma_3(P_n)$).
Note that 
$\eta_t$ acts by conjugation 
on $\overline{\sigma}^{-1}(\sn[3]\times 1\times \cdots \times 1)$ and $\overline{\sigma}^{-1}(1\times 1 \times 1\times \sn[n-3])$,
where $\overline{\sigma}$ is as in~\reqref{sespnquot}.
Further, by~\reqref{conjugalpha} and~\reqref{alphaijk},
we see that $\eta_t \alpha_{1,2,3} \eta_t^{-1}=\alpha_{1,2,3}$ in $B_n/\Gamma_3(P_n)$, and that if $\theta'$ is a word in the elements of $\mathcal{B}\setminus \{ \alpha_{1,2,3}\}$ then $\eta_t\theta'\eta_t^{-1}$, written in reduced form as a word in $P_n/\Gamma_3(P_n)$, does not contain $\alpha_{1,2,3}$.
Let $\theta=\alpha_{1,2,3}^l\theta_1$, where $l\in \Z$ and $\theta_1$ is a word in the elements of $\mathcal{B}\setminus \brak{\alpha_{1,2,3}}$. For all $i\in \brak{0,\ldots, t}$, we have: 
\begin{align*}
\delta_{3i,3}^{3}&=(\sigma_{3i+2}\sigma_{3i+1}^{-1})^{3}= \sigma_{3i+2}^{2} \sigma_{3i+2}^{-1}\sigma_{3i+1}^{-1}\sigma_{3i+2}\sigma_{3i+1}^{-1}\sigma_{3i+2}\sigma_{3i+1}^{-1}\\
&=\sigma_{3i+2}^{2} \sigma_{3i+1}\sigma_{3i+2}^{-1}\sigma_{3i+1}^{-2}\sigma_{3i+2}\sigma_{3i+1}^{-1} =\sigma_{3i+2}^{2} \sigma_{3i+1}^{2}\sigma_{3i+2}^{-2}\sigma_{3i+1}^{-2}= [\sigma_{3i+1}^{2}, \sigma_{3i+2}^{2}]^{-1}\\
&= [A_{3i+1,3i+2}, A_{3i+2,3i+3}]^{-1}=\alpha_{3i+1,3i+2,3i+3}^{-1},
\end{align*}
and since the $\delta_{3i,3}$ commute pairwise, it follows that $\eta_{t}^{3}=\alpha_{1,2,3}^{-1}\cdots \alpha_{3t+1,3t+2,3t+3}^{-1}$. Hence:
\begin{align*}
\rho^{3} &=(\theta\eta_t)^3 = \alpha_{1,2,3}^l\theta_1\eta_t \alpha_{1,2,3}^l\theta_1\eta_t^{-1}\eta_t^2 \alpha_{1,2,3}^l\theta_1\eta_t^{-2}\eta_t^3\\
& =\alpha_{1,2,3}^{3l}\theta_1 \ldotp \eta_t \theta_1\eta_t^{-1}\ldotp \eta_t^2 \theta_1\eta_t^{-2}\ldotp\alpha_{1,2,3}^{-1}\cdots \alpha_{3t+1,3t+2,3t+3}^{-1}
\end{align*}
in $\Gamma_2(P_n)/\Gamma_3(P_n)$.
As we saw above, the words $\eta_t \theta_1\eta_t^{-1}$ and $\eta_t^2 \theta_1\eta_t^{-2}$ (written in reduced form in $P_n/\Gamma_3(P_n)$) do not contain $\alpha_{1,2,3}$, and since $\rho^{3}=1$, it follows by comparing the coefficients of $\alpha_{1,2,3}$ that $3l=1$, which yields a contradiction. So $B_n/\Gamma_3(P_n)$ has no $3$-torsion as required.
\end{proof}

The following result is a consequence of Theorems~\ref{th:acgroups} and~\ref{th:almostb} and the definition of almost-Bieberbach groups.

\begin{cor}\label{cor:almostb}
Let $n\geq 3$ and let $k\geq 3$. Let $H$ be a subgroup of $\sn$ whose order is equal to $2^s3^t$ for some $s,t\in \N$. 
Then $\sigma^{-1}(H)/\Gamma_k(P_n)$ is an almost-Bieberbach group. In particular, the groups $B_3/\Gamma_3(P_3)$ and $B_4/\Gamma_3(P_4)$ are almost-Bieberbach groups of dimension $4$ and $10$ respectively.
\end{cor}

\begin{proof}
The first part is a consequence of Theorems~\ref{th:acgroups} and~\ref{th:almostb} and \redef{almostc}.
The second part follows by taking $H=\sn$ for $n\in\brak{3,4}$, and by applying the short exact sequence~\reqref{sespnquot} and the results of Table~\ref{table:table1}.
\end{proof}


\section{Torsion elements, conjugacy classes and a presentation of $B_n/\Gamma_3(P_n)$}
\label{sec:torsion}


Let $n\in \N$.
The main aim of this section is to prove Theorems~\ref{th:torsion} and~\ref{th:conjugacy} that describe the torsion elements of $B_n/\Gamma_3(P_n)$ and their conjugacy classes. We also exhibit a presentation of the quotient group $B_n/\Gamma_3(P_n)$ in \repr{nil}.

\subsection{A partition of the basis of $\Gamma_2(P_3)/\Gamma_3(P_3)$}\label{sec:partition}

Let $\delta_n$ be equal to the element $\delta_{0,n}$ of \req{eq14}, considered as an element of the group $B_n/\Gamma_3(P_n)$. By \relem{lem28}, $\delta_n^n\in \Gamma_2(P_n)/\Gamma_3(P_n)$.
We start by partitioning  the basis $\mathcal{B}$ of $\Gamma_2(P_n)/\Gamma_3(P_n)$ given in~\reqref{basis1} into orbits for the action by conjugation by $\delta_{n}$.

\begin{prop}\label{prop:basis}
Let $n\geq 5$. Then the basis $\mathcal{B}$ of $\Gamma_2(P_n)/\Gamma_3(P_n)$ given in \req{basis1} is invariant with respect to the action of $\Z_n$ given by conjugation by $\delta_n$. Further, under this action:
\begin{enumerate}[(a)]
\item\label{it:basisa} if $\operatorname{\text{gcd}}(n,3)=1$, $\mathcal{B}$ is the disjoint union of $\frac{(n-1)(n-2)}{6}$ orbits each of length $n$.

\item if $\operatorname{\text{gcd}}(n,3)\neq 1$, $\mathcal{B}$ is the disjoint union of $\frac{n(n-3)}{6}+1$ orbits, one of which is of length $\frac{n}{3}$, and the remaining $\frac{n(n-3)}{6}$ orbits are each of length $n$.

%
\end{enumerate}
\end{prop}

\begin{proof}
Let $n\geq 5$, and let $1\leq i<j<k\leq n$. Using~\reqref{conjugalpha},~\reqref{eq14} and~\reqref{alphaijk}, we see that:
\begin{equation}\label{eq:conjdelta}
\delta_n \alpha_{i,j,k} \delta_n^{-1}=\begin{cases}
\alpha_{i-1,j-1,k-1} & \text{if $i\geq 2$}\\
[A_{n,j-1},A_{j-1,k-1}]= [A_{j-1,n},A_{j-1,k-1}] =\alpha_{j-1,k-1,n} & \text{if $i=1$.}
\end{cases}
\end{equation}
 So the action of conjugation by $\delta_n$ permutes the elements of $\mathcal{B}$, which proves the first part of the statement. By~\reqref{conjdelta}, $\delta_n^n \alpha_{i,j,k} \delta_n^{-n}= \alpha_{i,j,k}$ for all $1\leq i<j<k\leq n$. Suppose that $1\leq l<n$ is such that $\delta_n^l \alpha_{i,j,k} \delta_n^{-l}= \alpha_{i,j,k}$. Then by~\reqref{conjdelta}, we have:
\begin{enumerate}[(i)]
\item\label{it:closei} $i-l+n=k$, $j-l=i$ and $k-l=j$, or
\item\label{it:closeii} $i-l+n=j$, $j-l+n=k$ and $k-l=i$.
\end{enumerate}
Summing the three equations in both cases, we see that $3\mid n$. In particular:
\begin{enumerate}[\textbullet]
\item if $3\nmid n$ then all of the orbits are of length $n$, and there are $\binom{n}{3}/n=\frac{(n-1)(n-2)}{6}$ orbits.
\item if $3\mid n$ then from~(\ref{it:closei}) and~(\ref{it:closeii}), we have either $l=\frac{n}{3}$ or $l=\frac{2n}{3}$. It follows that $j=i+\frac{n}{3}$ and $k=i+\frac{2n}{3}$, and for these values of $i,j$ and $k$, we obtain a single orbit of length $\frac{n}{3}$.  The remaining orbits are of length $n$, and so the number of such orbits is given by $(\binom{n}{3}-\frac{n}{3})/n$, which is equal to $\frac{n(n-3)}{6}$.\qedhere
\end{enumerate} 
\end{proof}


\begin{rem}
Let $n=3q+r$, where $q\in \N$ and $r\in \brak{0, 1,2}$. 
Let $\mathcal{S}$ be the subset of $\mathcal{B}$ defined by:
\begin{equation}\label{eq:transversal}
\!\!\!\mathcal{S}\!=\!
\begin{cases}
\!\setl{\alpha_{1,j,k}}{\text{$2\leq j \leq q+1$ and $2j-1\leq k\leq n-(j-1)$}} &\!\!\! \text{if $r\neq 0$}\\
\!\setl{\alpha_{1,j,k}}{\text{$2\leq j \leq q$ and $2j-1\leq k\leq n-(j-1)$}}\cup\{\alpha_{1, n/3+1, 2n/3+1}\} &\!\!\! \text{if $r=0$.}
\end{cases}
\end{equation}
Then $\mathcal{S}$ is a transversal for the
action of $\delta_n$ on $\mathcal{B}$ described in \repr{basis},
and $\mathcal{B}$ 
may be rewritten as:
\begin{equation*}
\mathcal{B}\!=\!\begin{cases}
\setl{\delta_n^{t} \alpha_{1,j,k} \delta_n^{-t}}{\text{$2\leq j \leq q+1$, $2j-1\leq k\leq n-(j-1)$ and $0\leq t\leq n-1$}} & \!\!\text{if $r\neq 0$}\\
& \\
\setl{\delta_n^{t} \alpha_{1,j,k} \delta_n^{-t}}{\text{$2\leq j \leq q$, $2j-1\leq k\leq n-(j-1)$ and $0\leq t\leq n-1$}}\bigcup\\ \hspace*{5.75cm}\setl{\delta_n^{t} \alpha_{1, n/3+1, 2n/3+1} \delta_n^{-t}}{0\leq t\leq n/3-1}
& \!\!\text{if $r=0$.}
\end{cases}
\end{equation*}
\end{rem}


With the notation of \repr{basis}, let $\mathcal{O}(n)$ be equal to the number of orbits of the action on $\mathcal{B}$ by conjugation by $\delta_{n}$, and let $\mathcal{T}=\brak{b_{i,1}}_{1\leq i\leq \mathcal{O}(n)}$ be a transversal for this action (for example, we may take $\mathcal{T}$ to be the transversal $\mathcal{S}$ defined in~\reqref{transversal}). We choose $\mathcal{T}$ so that $b_{\mathcal{O}(n),1}$ is a representative of the orbit of length $n/3$ if $3\mid n$. For $1\leq i\leq \mathcal{O}(n)$, let $q(i)$ be the length of the orbit of $b_{i,1}$, so that $q(i)=n/3$ if $3\mid n$ and $i=\mathcal{O}(n)$, and $q(i)=n$ otherwise, and let:
\begin{equation}\label{eq:defbij}
\text{$b_{i,j}=\delta_{n}^{j-1} b_{i,1} \delta_{n}^{-(j-1)}$ for all $1\leq j\leq q(i)$.}
\end{equation}
 Then
$\mathcal{B}=\setl{b_{i,j}}{\text{$1\leq i\leq \mathcal{O}(n)$ and $1\leq j\leq q(i)$}}$.

\subsection{A presentation of $B_n/\Gamma_3(P_n)$}\label{sec:presgamma3}

In this section, we exhibit a presentation of $B_n/\Gamma_3(P_n)$ by applying the techniques of~\cite[Proposition~1, p.~139]{Johnson} to obtain a presentation of a group extension to the short exact sequence~\reqref{gamma23P}.


\begin{prop}\label{prop:nil}
Let $n\geq 3$. 
\begin{enumerate}[(a)]
\item\label{it:relnsalphaa} The group  $P_n/\Gamma_3(P_n)$ has a presentation given by the generating set $\mathcal{X}_{n}=\brak{A_{i,j}}_{1\leq i<j\leq n}\cup \brak{\alpha_{r,s,t}}_{1\leq r<s<t\leq n}$, subject to the following relations:
\begin{enumerate}[(i)]
\item\label{it:relnsalpha1} the $\alpha_{r,s,t}$ commute pairwise and with the $A_{i,j}$.
\item\label{it:relnsalpha2} $
[A_{i,j}, A_{l,m}]= \begin{cases}
\alpha_{i,j,m} & \text{if $\operatorname{\text{Card}}\brak{i,j,l,m}=3$ and $j=l$}\\
\alpha_{i,k,m}^{-1} & \text{if $\operatorname{\text{Card}}\brak{i,j,l,m}=3$, where $k=j$ if $i=l$, and $k=l$ if $j=m$}\\
1 & \text{if $\operatorname{\text{Card}}\brak{i,j,l,m}\in \brak{2,4}$.}
\end{cases}$
\end{enumerate}
\item\label{it:relnsalphab} The group  $B_n/\Gamma_3(P_n)$ has a presentation given by the generating set $\brak{\sigma_{k}}_{1\leq k\leq n-1}\cup \mathcal{X}_{n}$, 
subject to the following relations:
\begin{enumerate}[(i)]
\item relations~(\ref{it:relnsalphaa})(\ref{it:relnsalpha1}) and~(\ref{it:relnsalphaa})(\ref{it:relnsalpha2}) emanating from those of $P_n/\Gamma_3(P_n)$.
\item the Artin braid relations~\reqref{artin1}, viewed in $B_n/\Gamma_3(P_n)$.
\item the conjugacy relations described in~\reqref{action1} and~\reqref{action2}.
\end{enumerate}
\end{enumerate}
\end{prop}

\begin{proof}
To prove parts~(\ref{it:relnsalphaa}) and~(\ref{it:relnsalphab}), it suffices to apply~\cite[Proposition~1, p.~139]{Johnson} first to~\reqref{gamma23P} taking $k=2$, and then to~\reqref{sespnquot} taking $k=3$ using also part~(\ref{it:relnsalphaa}).
\end{proof}

\subsection{Torsion elements and conjugacy classes in $B_n/\Gamma_3(P_n)$}\label{sec:torsion2}



In this section, we prove Theorems~\ref{th:torsion} and~\ref{th:conjugacy}. We start by proving the following lemma whose result generalises~\cite[Theorem~3(a)]{GGO}.

\begin{lem}\label{lem:1to1}
Let $m,k \in \N$. 
\begin{enumerate}
\item\label{it:1to1a} Let $n\in \N$ be such that $2\leq n\leq m$, and let $\map{\iota}{B_{n}}[B_m]$ denote the 
injective homomorphism defined by $\iota(\sigma_i)=  \sigma_i$ for all $1\leq i\leq n-1$.
Then the induced homomorphism $\map{\overline{\iota}}{B_n/\Gamma_k(P_n)}[B_m/\Gamma_k(P_m)]$ is injective. 
\item\label{it:1to1b} Let $t\in \N$, let $n_1,n_2,\ldots,n_t$ be integers greater than or equal to $2$ for which $\sum_{i=1}^{t} \, n_i\leq m$, and let $\map{\zeta}{B_{n_1} \times \cdots \times B_{n_t}}[B_m]$ denote the natural inclusion. Then the induced homomorphism $\map{\overline{\zeta}}{B_{n_1}/\Gamma_k(P_{n_1}) \times \cdots \times B_{n_t}/\Gamma_k(P_{n_t})}[B_m/\Gamma_k(P_m)]$ is injective.
\end{enumerate}
\end{lem}


\begin{proof}
Let $t$, $m$ and $n$ be positive integers such that $2\leq n\leq m$.
\begin{enumerate}
\item 
 If $k=1$, the result is straightforward. So assume that $k\geq 2$. For $2\leq n\leq m$, let $\map{\iota_{n,m}}{B_{n}}[B_m]$ (resp.\ $\map{\overline{\iota}_{n,m,k}}{B_n/\Gamma_k(P_n)}[B_m/\Gamma_k(P_m)]$) be the homomorphism $\iota$ (resp.\ $\overline{\iota}$) given in the statement. By \req{defaij}, $\iota_{n,m}$ restricts to an injective homomorphism $\map{\iota_{n,m}\left\lvert_{P_{n}}\right.}{P_{n}}[P_{m}]$ given by $\iota_{n,m}\left\lvert_{P_{n}}\right.(A_{i,j})=A_{i,j}$ for all $1\leq i<j\leq n$. For $n\leq q\leq m-1$, let $\map{\iota_{q}}{P_q}[P_{q+1}]$ denote the homomorphism $\iota_{q,q+1}\left\lvert_{P_{q}}\right.$. We claim that the homomorphism $\map{\overline{\iota}_{n,m}\left\lvert_{\Gamma_k(P_n)/\Gamma_{k+1}(P_n)}\right.}{\Gamma_k(P_n)/\Gamma_{k+1}(P_n)}[\Gamma_k(P_m)/\Gamma_{k+1}(P_m)]$ that is induced by $\iota_{n,m}\left\lvert_{P_{n}}\right.$ is injective for all $k\geq 2$. To see this, first note that since $\iota_{n,m}\left\lvert_{P_{n}}\right.$ is equal to the composition $\iota_{m-1}\circ \cdots\circ \iota_{n}$, the homomorphism $\overline{\iota}_{n,m}\left\lvert_{\Gamma_k(P_n)/\Gamma_{k+1}(P_n)}\right.$ is equal to the composition $\overline{\iota}_{m-1}\left\lvert_{\Gamma_k(P_{m-1})/\Gamma_{k+1}(P_{m-1})}\right.\circ \cdots\circ \overline{\iota}_{n}\left\lvert_{\Gamma_k(P_n)/\Gamma_{k+1}(P_n)}\right.$, where the homomorphism
\begin{equation*}
 \map{\overline{\iota}_{q}\left\lvert_{\Gamma_k(P_q)/\Gamma_{k+1}(P_q)}\right.}{\Gamma_k(P_q)/\Gamma_{k+1}(P_q)}[\Gamma_k(P_{q+1})/\Gamma_{k+1}(P_{q+1})]
\end{equation*}
is induced by $\iota_q$ for all $n\leq q \leq m-1$, it suffices to prove that $\overline{\iota}_{q}\left\lvert_{\Gamma_k(P_q)/\Gamma_{k+1}(P_q)}\right.$ is injective. To do so, consider the Fadell-Neuwirth short exact sequence $1\to \ker{p} \to P_{q+1} \stackrel{p}{\to} P_q \to 1$, where $p$ is the homomorphism given geometrically by forgetting the last string. Using the presentation given in \req{purerelations}, we obtain two well-known facts, first that $\iota_q$ is a section for $p$, and secondly that the resulting semi-direct product is almost direct, \emph{i.e.}\ the action induced by $P_q$ on the Abelianisation $\ker{p}/\Gamma_2(\ker{p})$ of $\ker{p}$ is trivial. It follows from~\cite[Theorem~3.1]{FRinv} that the induced sequence of homomorphisms $1\to \Gamma_k(\ker{p})/\Gamma_{k+1}(\ker{p}) \to \Gamma_k(P_{q+1})/\Gamma_{k+1}(P_{q+1}) \stackrel{\overline{p}}{\to} \Gamma_k(P_q)/\Gamma_{k+1}(P_q) \to 1$ is split short exact, and that the homomorphism $\overline{\iota}_{q}\left\lvert_{\Gamma_k(P_q)/\Gamma_{k+1}(P_q)}\right.$ is a section for $\overline{p}$. In particular, $\overline{\iota}_{q}\left\lvert_{\Gamma_k(P_q)/\Gamma_{k+1}(P_q)}\right.$ is injective, and the claim follows.
Now consider the following commutative diagram of short exact sequences:
\begin{equation*}
\begin{xy}*!C\xybox{%
\xymatrix{%
1 \ar[r] & \Gamma_k(P_n)/\Gamma_{k+1}(P_n) \ar[r] \ar[d]_{\overline{\iota}_{n,m}\left\lvert_{\Gamma_k(P_n)/\Gamma_{k+1}(P_n)}\right.} & B_n/\Gamma_{k+1}(P_n) \ar[r] \ar[d]^{\overline{\iota}_{n,m,k+1}} &  B_n/\Gamma_{k}(P_n) \ar[r] \ar[d]^{\overline{\iota}_{n,m,k}} & 1\\
1 \ar[r] & \Gamma_k(P_m)/\Gamma_{k+1}(P_m) \ar[r] & B_m/\Gamma_{k+1}(P_m) \ar[r] &  B_m/\Gamma_{k}(P_m) \ar[r] & 1,}}
\end{xy}
\end{equation*}
where the rows are the short exact sequences given by the central extension~\reqref{gammaquot}. The statement of part~(\ref{it:1to1a}) is then a consequence of applying the $5$-Lemma to this diagram, the above claim, induction on $k$, and the fact that $\map{\overline{\iota}}{B_n/\Gamma_2(P_n)}[B_m/\Gamma_2(P_m)]$ is an injective homomorphism by~\cite[Theorem~3(a)]{GGO}.

\item Let $n_{1},\ldots,n_{t}$ be integers greater than or equal to $2$ such that $\sum_{i=1}^{t}\, n_{i}\leq m$, let $\map{\sigma}{B_{\sum_{i=1}^{t}\, n_{i}}}[{\sn[\sum_{i=1}^{t}\, n_{i}]}]$ denote the usual homomorphism that to a braid associates its permutation, and let $B_{n_{1},\ldots,n_{t}}$ denote the corresponding mixed braid group, namely the preimage under $\sigma$ of the subgroup $\sn[n_{1}] \times \cdots \times \sn[n_{t}]$ of $\sn[\sum_{i=1}^{t}\, n_{i}]$.  For $1\leq i\leq t$, let $\map{\phi_{i}}{B_{n_{i}}}[B_{n_{1},\ldots,n_{t}}]$ denote the embedding of $B_{n_{i}}$ into the $i\up{th}$ factor of $B_{n_{1},\ldots,n_{t}}$. Note that $\zeta$ is the equal to the following composition:
\begin{equation*}
B_{n_1} \times \cdots \times B_{n_t} \!\xrightarrow{\phi_{1}\! \times\cdots\times \phi_{t}} B_{n_1,n_2,\ldots,n_t} \!\lhra \!B_{\sum_{i=1}^{t}n_i} \!\lhra\! B_{m}.
\end{equation*}
Note that $\phi_{i}$ induces a homomorphism $\map{\overline{\phi_{i}}}{B_{n_i}/\Gamma_k(P_{n_i})}[B_{n_1,n_2,\ldots,n_t}/\Gamma_k\bigl(P_{\sum_{i=1}^{t}n_i}\bigr)]$ because $\phi_{i}(\Gamma_k(P_{n_{i}}))\subset \Gamma_k\bigl(P_{\sum_{i=1}^{t}n_i}\bigr)$. Now let $\map{\psi_{i}}{B_{n_{1},\ldots,n_{t}}}[B_{n_i}/\Gamma_k(P_{n_i})]$ be the composition of the projection onto the $i\up{th}$ factor of $B_{n_{1},\ldots,n_{t}}$, followed by the canonical projection $B_{n_i}\to B_{n_i}/\Gamma_k(P_{n_i})$. Under this composition, the normal subgroup $P_{\sum_{i=1}^{t}n_i}$ of $B_{n_{1},\ldots,n_{t}}$ is sent to $P_{n_i}/\Gamma_k(P_{n_i})$, hence the normal subgroup $\Gamma_k\bigl(P_{\sum_{i=1}^{t}n_i}\bigr)$ of $B_{n_{1},\ldots,n_{t}}$ is sent to the trivial element of $B_{n_i}/\Gamma_k(P_{n_i})$, from which it follows that $\psi_{i}$ induces a homomorphism $\map{\overline{\psi_{i}}}{B_{n_{1},\ldots,n_{t}}/\Gamma_k\bigl(P_{\sum_{i=1}^{t}n_i}\bigr)}[B_{n_i}/\Gamma_k(P_{n_i})]$. From the constructions of $\phi_{i}$ and $\psi_{i}$, we see that $\overline{\psi_{i}}\circ \overline{\phi_{i}}=\id_{B_{n_i}/\Gamma_k(P_{n_i})}$ for all $1\leq i\leq t$, and so the composition
\begin{equation*}
\frac{B_{n_1}}{\Gamma_k(P_{n_1})}  \times \cdots \times \frac{B_{n_t}}{\Gamma_k(P_{n_t})} \!\xrightarrow{\overline{\phi_{1}} \times\cdots\times \overline{\phi_{t}}}
\frac{B_{n_1,n_2,\ldots,n_t}}{\Gamma_k\bigl(P_{\sum_{i=1}^{t}n_i}\bigr)} \xrightarrow{\overline{\psi_{1}}  \times\cdots\times \overline{\psi_{t}}} \frac{B_{n_1}}{\Gamma_k(P_{n_1})}  \times \cdots \times \frac{B_{n_t}}{\Gamma_k(P_{n_t})}
\end{equation*}
is the identity. Thus $\overline{\phi_{1}}\times\cdots\times\overline{\phi_{t}}$ is injective, and the composition
\begin{equation*}
\frac{B_{n_1}}{\Gamma_k(P_{n_1})} \times \cdots \times \frac{B_{n_t}}{\Gamma_k(P_{n_t})} \!\xrightarrow{\overline{\phi_{1}} \times\cdots\times \overline{\phi_{t}}} \frac{B_{n_1,n_2,\ldots,n_t}}{\Gamma_k\bigl(P_{\sum_{i=1}^{t}n_i}\bigr)} \to \frac{B_{\sum_{i=1}^{t}n_i}}{\Gamma_k\bigl(P_{\sum_{i=1}^{t}n_i}\bigr)} \to \frac{B_{m}}{\Gamma_k(P_{m})},
\end{equation*}
which is the homomorphism $\overline{\zeta}$ of the statement, may be seen to be injective using part~(\ref{it:1to1a}) and the injectivity of the homomorphism $B_{n_{1},\ldots,n_{t}}\lhra B_{\sum_{i=1}^{t}\, n_{i}}$. \qedhere
\end{enumerate}
\end{proof}

The following results allow us to decide whether certain elements of a group are of finite order or not. They will be used to obtain finite-order elements in $B_{n}/\Gamma_{3}(P_{n})$, as well as in the proof of \reth{torsion}.


\begin{lem}\label{lem:nodtorsion}
Let $d\geq 2$, let $G$ be a group that has no $d$-torsion, and let $\alpha\in G$. Suppose that $\map{\sigma}{G}[\sn]$ is a surjective homomorphism whose kernel $K$ is torsion free, and such that $d$ divides the order of $\sigma(\alpha)$. Then for all $\theta\in K$, $\theta \alpha$ is of infinite order in $G$.
\end{lem}

\begin{proof}
Suppose on the contrary that there exists $\theta\in K$ for which $\theta\alpha$ is of finite order, $m$ say, in $G$. Let $q$ denote the order of $\sigma(\alpha)$. Since $\theta\in K$, $\id= \sigma((\theta\alpha)^{m})= (\sigma(\alpha))^{m}$, and thus $q\mid m$. In a similar manner, we see that $\sigma((\theta\alpha)^{q})=\id$, so $(\theta\alpha)^{q}\in K$. Now $1=(\theta\alpha)^{m}=((\theta\alpha)^{q})^{m/q}$, and the fact that $K$ is torsion free implies that $(\theta\alpha)^{q}=1$, hence $q=m$. By hypothesis, $d\mid q$, and so $((\theta\alpha)^{q/d})^{d}=1$. This implies that the order of $(\theta\alpha)^{q/d}$ divides $d$. On the other hand, the fact that $\sigma(\theta\alpha)$ is of order $q$ implies that $\sigma((\theta\alpha)^{q/d})$ is of order $d$, and so the order of $(\theta\alpha)^{q/d}$ cannot be strictly less than $d$. Hence $(\theta\alpha)^{q/d}$ is of order $d$, but this contradicts the fact that $G$ has no $d$-torsion.
\end{proof}


\begin{prop} \label{prop:deltan}  Let
$n\geq 3$.
\begin{enumerate}[(a)]
\item\label{it:deltana} If $n$ is even then for all $k\in \N$ and $\theta\in P_{n}/\Gamma_3(P_n)$, $(\theta\delta_{n})^{k} \notin \Gamma_2(P_n)/\Gamma_3(P_n)$.
\item\label{it:deltanb} Suppose that $n$ is odd. Then the element $\delta_n^n$ belongs to $\Gamma_2(P_n)/\Gamma_3(P_n)$, and with the notation of~\reqref{defbij}, we have: 
\begin{equation}\label{eq:deltann}
\delta_n^n=\prod_{i=1}^{\mathcal{O}(n)} \prod_{j=1}^{q(i)} b_{i,j}^{m_{i,j}},
\end{equation}
where $m_{i,1}=m_{i,2}=\cdots=m_{i,q(i)}$ for all $1\leq i\leq \mathcal{O}(n)$. Further, if $3 \mid n$ then for all $k\in \N$ and $\theta\in P_{n}/\Gamma_3(P_n)$, $(\theta\delta_{n})^{k}$ is non trivial in $B_{n}/\Gamma_3(P_n)$.
\end{enumerate}
\end{prop} 
 
\begin{proof}\mbox{} 
\begin{enumerate}[(a)]
\item Let $n$ be even, and assume on the contrary that $(\theta\delta_{n})^{k} \in \Gamma_2(P_n)/\Gamma_3(P_n)$ for some $k\in \N$ and $\theta\in P_{n}/\Gamma_3(P_n)$. Let $\map{f}{B_{n}/\Gamma_3(P_n)}[B_{n}/\Gamma_2(P_n)]$ be the homomorphism given in~\reqref{gammaquot} with $j=3$. Then $(\overline{\theta} \overline{\delta}_{n})^{k}=1$ in $B_{n}/\Gamma_2(P_n)$, where $\overline{\theta}=f(\theta)$ and $\overline{\delta}_{n}=f(\delta_{n})$. We now apply \relem{nodtorsion} to the short exact sequence~\reqref{sespnquot}, taking $k=2$, $G=B_{n}/\Gamma_2(P_n)$, $\alpha=\overline{\delta}_{n}$, and the homomorphism $\sigma$ of that lemma to be $\overline{\sigma}$. Note that $\overline{\sigma}(\overline{\delta}_{n})$ is of order $n$, so $2\mid n$, and $B_{n}/\Gamma_{2}(P_{n})$ has no $2$-torsion by~\cite[Theorem~2]{GGO}. Since $P_{n}/\Gamma_{2}(P_{n})$ is torsion free, it follows from \relem{nodtorsion} that $\overline{\theta} \overline{\delta}_{n}$ is of infinite order in $B_{n}/\Gamma_{2}(P_{n})$, and we obtain a contradiction. We conclude that $(\theta\delta_{n})^{k} \notin \Gamma_2(P_n)/\Gamma_3(P_n)$.


\item Since $\overline{\delta}_n^n=1$ in $B_n/\Gamma_2(P_n)$ by \relem{lem28}, it follows from the short exact sequence~\reqref{gammaquot} with $j=3$ that $\delta_n^n\in \Gamma_2(P_n)/\Gamma_3(P_n)$. 
Hence in terms of the basis $\mathcal{B}$ of $\Gamma_2(P_n)/\Gamma_3(P_n)$ given by~\reqref{basis1}, and with the notation of~\reqref{defbij}, for all $1\leq i\leq \mathcal{O}(n)$ and $1\leq j\leq q(i)$, there exist $m_{i,j}\in \Z$ that are unique and for which~\reqref{deltann} holds. The equality $m_{i,1}=m_{i,2}=\cdots=m_{i,q(i)}$ follows by conjugating~\reqref{deltann} by $\delta_{n}$ and using~\reqref{defbij}. The last part of the statement follows by applying \relem{nodtorsion} to the short exact sequence~\reqref{sespnquot}, where we take $G=B_{n}/\Gamma_3(P_n)$, $d=3$ and $\alpha=\delta_{n}$, and using \reth{almostb}.\qedhere
\end{enumerate}
%
%
\end{proof}


\begin{thm}\label{th:ntorsion} 
Let $n\in \N$, let $m\leq n$, and let $s$ be the largest divisor of $m$ for which $\operatorname{\text{gcd}}(s,6)=1$.
If $s>1$, the group $B_n/\Gamma_3(P_n)$ possesses
infinitely-many elements of order $s$.
\end{thm}

\begin{proof}
Let $1\leq m\leq n$, and let $s$ be the largest divisor of $m$ for which $\operatorname{\text{gcd}}(s,6)=1$. Since we assume that $s>1$, we must have $n\geq 5$. 
Further, the fact that $s\leq n$ implies using \relem{1to1}(\ref{it:1to1a}) that 
the homomorphism $\map{\overline{\iota}}{B_s/\Gamma_3(P_s)}[B_n/\Gamma_3(P_n)]$ is injective. So it suffices to prove that 
if $n\in \N$ is relatively prime with $6$ then $B_n/\Gamma_3(P_n)$ possesses elements of order $n$.
In this case, $\delta_{n}^{n}$ is as given in~\reqref{deltann}, where $\mathcal{O}(n)=(n-1)(n-2)/6$ by \repr{basis}(\ref{it:basisa}), and $q(i)=n$ for all $1\leq i\leq \mathcal{O}(n)$. 
Let:
\begin{equation}\label{eq:defdeltanhat}
\widehat{\delta}_{n}=\theta \delta_{n},
\end{equation}
where we take $\theta$ to be an element of $\Gamma_2(P_n)/\Gamma_3(P_n)$, so $\theta=\prod_{i=1}^{\mathcal{O}(n)} \prod_{j=1}^{n} b_{i,j}^{r_{i,j}}$, where $r_{i,j}\in \Z$ for all $1\leq i\leq \mathcal{O}(n)$ and $1\leq j\leq n$. Then by~\reqref{defbij} and \repr{deltan}(\ref{it:deltanb}), we have:
\begin{align}
\widehat{\delta}_{n}^{\,n} & = (\theta \delta_{n})^{n}=\theta\ldotp \delta_{n}\theta \delta_{n}^{-1} \ldotp \delta_{n}^{2}\theta \delta_{n}^{-2} \cdots 
\delta_{n}^{n-1}\theta \delta_{n}^{-(n-1)} \ldotp \delta_{n}^{n}\notag\\
&= \prod_{i=1}^{\mathcal{O}(n)} \Bigl(\prod_{j=1}^{n} b_{i,j}\Bigr)^{\sum_{j=1}^{n} r_{i,j}} \ldotp \prod_{i=1}^{\mathcal{O}(n)} \prod_{j=1}^{n} b_{i,j}^{m_{i,1}} =\prod_{i=1}^{\mathcal{O}(n)} \Bigl(\prod_{j=1}^{n} b_{i,j}\Bigr)^{\sum_{j=1}^{n} r_{i,j}+m_{i,1}}.\label{eq:gendeltann}
\end{align}
Taking $k=3$ in~\reqref{sespnquot}, $\overline{\sigma}(\widehat{\delta}_{n})$ is of order $n$ in $\sn$, and it follows from this and~\reqref{gendeltann} that $\widehat{\delta}_{n}$ is of order $n$ if and only if:
\begin{equation}\label{eq:compat}
\text{$\sum_{j=1}^{n} r_{i,j}=-m_{i,1}$ for all $i=1,\ldots, \mathcal{O}(n)$.}
\end{equation}
This system of equations has infinitely-many solutions in the $r_{i,j}$, which yields infinitely-many elements in $B_n/\Gamma_3(P_n)$ of order $n$. 
\end{proof}

\begin{rem}
If $n\in \N$ and $\operatorname{\text{gcd}}(n,6)=1$, 
the proof of \reth{ntorsion} shows how to obtain explicit finite-order elements of $B_n/\Gamma_3(P_n)$.
However, as part of the process, we need to determine $\delta_n^n$ in terms of the elements of the basis $\mathcal{B}$. This seems to be an arduous computation in general, as the calculation given in \resec{gamma5} in the case $n=5$ indicates. 
%
\end{rem}

We are now able to prove \reth{torsion}.


\begin{proof}[Proof of \reth{torsion}]
Let $n\geq 5$, and let $\tau\in \N$ be such that $\gcd(\tau,6)=1$. If $\beta$ is an element of $B_n/\Gamma_3(P_n)$ of order $\tau$ then using~\reqref{sespnquot} and the fact that $P_n/\Gamma_3(P_n)$ is torsion free by \relem{ptorsion}(\ref{it:ptorsiona}), it follows that $\overline{\sigma}(\beta)$ is an element of $\sn$ of order $\tau$. Conversely, suppose that $x\in \sn$ is an element of order $\tau$, and let $\eta_{1}\cdots \eta_{t}$ be the cycle decomposition of $x$. Then $\tau=\operatorname{\text{lcm}}(n_{1},\ldots, n_{t})$, where for $i=1,\ldots, t$, $\eta_{i}$ is of length $n_{i}$. In particular, $\gcd(n_{i},6)=1$ for all $i=1,\ldots, t$, and $\sum_{i=1}^{t}\, n_{i} \leq n$ by~\cite{Ho}. Let $y\in \sn$ be such that
\begin{equation}\label{eq:permyxy}
\!\!\!yxy^{-1}= (1,\ldots,n_1)(n_1+1,\ldots,n_1+n_2) \cdots (n_1+\cdots+n_{t-1}+1,\ldots, n_1+\cdots+n_{t}).
\end{equation}
Since $\overline{\sigma}$ is surjective, there exists $\rho\in B_n/\Gamma_3(P_n)$ such that $\overline{\sigma}(\rho)=y$. For $i=1,\ldots, t$, let $\widehat{\delta}_{n_i}\in B_{n_i}/\Gamma_3(P_{n_i})$ be as defined in~\reqref{defdeltanhat}, where the coefficients of $\theta$ satisfy~\reqref{compat}. From the proof of \reth{ntorsion}, $\widehat{\delta}_{n_i}$ is of order $n_i$. Let $\beta=\overline{\zeta}(\widehat{\delta}_{n_1},\ldots, \widehat{\delta}_{n_t})$ in $B_n/\Gamma_3(P_{n})$. By taking $k=3$ in \relem{1to1}(\ref{it:1to1b}), we see that $\beta$ is of order $\tau$, and since the permutation associated to $\widehat{\delta}_{n_i}$ is the $n_{i}$-cycle $(1,\ldots,n_{i})$, it follows that $\overline{\sigma}(\beta)=yxy^{-1}$, and hence $\overline{\sigma}(\alpha)=x$, where $\alpha=\rho^{-1}\beta\rho$. In particular, $B_n/\Gamma_3(P_n)$ has elements of order $\tau$ if and only if $\sn$ does.
\end{proof}

One consequence of \reth{torsion} is the classification of the isomorphism classes of the finite cyclic subgroups of $B_n/\Gamma_3(P_n)$. Using \relem{1to1}(\ref{it:1to1b}), we may also caracterise the isomorphism classes of the finite Abelian subgroups of $B_n/\Gamma_3(P_n)$ in a manner similar to that of $B_n/\Gamma_2(P_n)$ given in~\cite[Theorem~6]{GGO}.

\begin{cor}\label{cor:existAb}
Let $n\geq 3$. Then there is a one-to-one correspondence between the isomorphism classes of the finite Abelian subgroups of $B_n/\Gamma_3(P_n)$ and those of the finite Abelian subgroups of $\sn$ whose order is relatively prime with $6$. 
\end{cor}

\begin{proof}
The proof is similar to that of~\cite[Theorem~6]{GGO}, using \relem{1to1}(\ref{it:1to1b}) and the finite-order elements of the form~\reqref{defdeltanhat} constructed in the proof of \reth{ntorsion}.
\end{proof}

To end this section, we prove \reth{conjugacy}, which says that if $n\geq 5$, two finite-order elements of  $B_n/\Gamma_3(P_n)$ are conjugate if and only if their associated permutations are conjugate.

\begin{proof}[Proof of \reth{conjugacy}]
Let $n\geq 5$, and let $\alpha$ and $\beta$ be two finite-order elements of $B_n/\Gamma_3(P_n)$ whose associated permutations have the same cycle type. Conjugating $\beta$ in $B_n/\Gamma_3(P_n)$ if necessary, we may suppose that $\beta$ is as in the proof of \reth{torsion}, in particular, it is of order $\tau=\operatorname{\text{lcm}}(n_1,\ldots,n_t)$, where $\operatorname{\text{gcd}}(n_i,6)=1$ for all $i=1,\ldots,t$, and $\overline{\sigma}(\beta)$ is equal to the permutation given in~\reqref{permyxy}.
To complete the proof of the theorem, it suffices to show that if $\alpha\in B_n/\Gamma_3(P_n)$ is any finite-order element such that $\overline{\sigma}(\alpha)$ has the same cycle type as $\overline{\sigma}(\beta)$, then $\alpha$ and $\beta$ are conjugate in $B_n/\Gamma_3(P_n)$. Conjugating $\alpha$ if necessary, we may suppose further that $\overline{\sigma}(\alpha)=\overline{\sigma}(\beta)$. Since $\ker{\overline{\sigma}}=P_{n}/\Gamma_3(P_n)$ is torsion free by~\relem{ptorsion}(\ref{it:ptorsiona}), it follows that $\alpha$ is also of order $\tau$.

Let $\mathcal{B}'$ be the union of the elements of the basis $\mathcal{B}$ defined in~\reqref{basis1} and their inverses. From~\reqref{conjugalpha} and~\reqref{alphaijk}, $\beta$ acts on $\mathcal{B}'$ by conjugation, and if $\beta \alpha_{i,j,k} \beta^{-1}= \alpha_{r,s,t}^{-1}$ for some $1\leq i<j<k\leq n$ and $1\leq r<j<s\leq t$, then either $i=n_1+\cdots +n_m+1<j\leq n_1+\cdots +n_{m+1}<k$ for some $0\leq m< t$, or $i< j=n_1+\cdots +n_m+1<k\leq n_1+\cdots +n_{m+1}$ for some $1\leq m< t$. From this and the fact that $n_l$ is odd for all $l=1,\ldots, t$, we see that under this action, the orbits of $\alpha_{i,j,k}$ and $\alpha_{i,j,k}^{-1}$ are disjoint for all $1\leq i<j<k\leq n$, and that the orbit of $\alpha_{i,j,k}^{-1}$ is obtained by taking inverses of the elements of the orbit of $\alpha_{i,j,k}$. We then choose a transversal $\bigl\{ \bigl. b_{i,1},b_{i,1}^{-1}\bigr\rvert 1\leq i\leq q\bigr\}$ for this action on $\mathcal{B}'$, where for $1\leq i\leq q$, $b_{i,1}\in \mathcal{B}$, and
we let $s_i$ denote the length of the orbit of $b_{i,1}$. Observe that $s_i$ divides $\tau$. Then $\mathcal{B}''=\setl{b_{i,j}}{\text{$1\leq i\leq q$ and $1\leq j\leq s_i$}}$ is a basis of $\Gamma_2(P_n)/\Gamma_3(P_n)$, where:
\begin{equation}
\text{$b_{i,j}=\beta^{j-1}b_{i,1} \beta^{-(j-1)}$ for all $1\leq i\leq q$ and $1\leq j\leq s_i$.}
\end{equation}

As in the proof of \repr{deltan}(\ref{it:deltana}), let $\map{f}{B_n/\Gamma_3(P_n)}[B_n/\Gamma_2(P_n)]$ be the projection given in~\reqref{gammaquot} with $j=3$, and let $\map{\overline{\sigma}'}{B_n/\Gamma_2(P_n)}[\sn]$ be the homomorphism given in~\reqref{sespnquot} with $k=2$. Since $\ker{f}=\Gamma_2(P_n)/\Gamma_3(P_n)$ is torsion free, $f(\alpha)$ and $f(\beta)$ are elements of $B_n/\Gamma_2(P_n)$ of order $\tau$. Further, $\overline{\sigma}=\overline{\sigma}'\circ f$, so $\overline{\sigma}'(f(\alpha))=\overline{\sigma}'(f(\beta))$, and applying~\cite[Theorem~5]{GGO}, there exists $\xi \in B_n/\Gamma_2(P_n)$ such that $\xi f(\alpha)\xi^{-1}= f(\beta)$. Since $f$ is surjective, there exists $\xi'\in B_n/\Gamma_3(P_n)$ such that $f(\xi' \alpha\xi'^{-1})= f(\beta)$. So conjugating $\alpha$ if necessary, there exists $\theta \in \Gamma_2(P_n)/\Gamma_3(P_n)$ such that $\theta\beta= \alpha$.

It suffices to show that there exists $\Omega \in \Gamma_2(P_n)/\Gamma_3(P_n)$ such that $\Omega \alpha\Omega^{-1}=\beta$, or equivalently that:
\begin{equation}\label{eq:Omegaconj}
\theta=(\beta \Omega \beta^{-1})\Omega^{-1},
\end{equation}
using the fact that $\theta$ commutes with $\Omega$. Let $\theta=\prod_{\substack{1\leq i\leq q\\ 1\leq j\leq s_{i}}}b_{i,j}^{r_{i,j}}$, and 
$\Omega=\prod_{\substack{1\leq i\leq q\\ 1\leq j\leq s_{i}}}b_{i,j}^{x_{i,j}}$, where $r_{i,j},x_{i,j}\in \Z$ for all $1\leq i\leq q$ and $1\leq j\leq s_{i}$. 
Since $\theta\beta=\alpha$, the elements $\beta$ and $\theta\beta$ are of order $\tau$, and so:
\begin{equation}\label{eq:conjtheta}
1=(\theta\beta)^{\tau}=\theta (\beta\theta\beta^{-1}) \cdots (\beta^{\tau-1}\theta\beta^{-(\tau-1)}) \beta^{\tau}= \theta (\beta\theta\beta^{-1}) \cdots (\beta^{\tau-1}\theta\beta^{-(\tau-1)}).
\end{equation}
From the construction of the basis $\mathcal{B}''$, it follows from~\reqref{conjtheta} that:
\begin{equation*}
1=\theta (\beta\theta\beta^{-1}) \cdots (\beta^{\tau-1}\theta\beta^{-(\tau-1)})= \prod_{1\leq i\leq q} \biggl(\prod_{1\leq j\leq s_{i}} b_{i,j} \biggr)^{\tau(\sum_{1\leq j\leq s_{i}}r_{i,j})/s_i},
\end{equation*}
from which we conclude that $\sum_{1\leq j\leq s_{i}}r_{i,j}=0$ for all $1\leq i\leq q$. In a similar manner,~\reqref{Omegaconj} may be written as:
\begin{equation*}
\prod_{\substack{1\leq i\leq q\\ 1\leq j\leq s_{i}}}b_{i,j}^{r_{i,j}}= \prod_{\substack{1\leq i\leq q\\ 1\leq j\leq s_{i}}}b_{i,j}^{x_{i,j-1}-x_{i,j}},
\end{equation*}
where the index $j-1$ of $x_{i,j-1}$ is taken modulo $s_i$. So for all $i=1,\ldots,q$, we obtain a system of equations $x_{i,j-1}-x_{i,j}=r_{i,j}$ for all $1\leq j\leq s_i$ that is subject to the compatibility condition $\sum_{1\leq j\leq s_{i}}r_{i,j}=0$, and it may be seen easily that each such system admits a solution. Using~\reqref{Omegaconj}, we conclude that $\alpha$ and $\beta$ are conjugate as required.
\end{proof}

\section{Some examples with a small numbers of strings}\label{sec:small}
  
%


In this section, we study a couple of examples where the number of strings is small. In \resec{gamma3}, we determine, up to conjugacy, the almost-Bierberbach subgroups of $B_3/\Gamma_3(P_3)$ that contain $P_3/\Gamma_3(P_3)$, and we identify them using the classification of~\cite{Dekimpe}. In \resec{gamma5},  we calculate explicitly $\delta_5^5$ in $\Gamma_2(P_5)/\Gamma_3(P_5)$ in terms of the basis $\mathcal{B}$. This example illustrates the computational difficulties that we encounter with respect to the constructions of \repr{deltan} and \reth{ntorsion}. We finish the paper with a remark concerning the study of the quotients $B_n/\Gamma_k(P_n)$ for $k>3$.

\subsection{Some almost-Bieberbach subgroups of $B_3/\Gamma_3(P_3)$}\label{sec:gamma3}
  
In this section, we describe the almost-Bieberbach groups $\sigma^{-1}(H)/\Gamma_3(P_3)$, 
where $H$ is a subgroup of $\sn[3]$. Recall that representatives of the conjugacy classes of subgroups of $\sn[3]$ are given by $\brak{\id}$, $\ang{(1,2)}$, $\ang{(1,2,3)}$ and $\sn[3]$. 
As we shall see now in \reth{4dimab}, these groups are of dimension $4$ and their holonomy group is $H$, and that any subgroup of $B_3/\Gamma_3(P_3)$ containing $P_3/\Gamma_3(P_3)$ is in fact of the form $\sigma^{-1}(H)/\Gamma_3(P_3)$, where  $H$ is a subgroup of $\sn[3]$.

\begin{thm}\label{th:4dimab}\mbox{}
\begin{enumerate}
\item\label{it:4dimab0} 
Let $K$ be a subgroup 
of $B_3/\Gamma_3(P_3)$ that contains $P_3/\Gamma_3(P_3)$.
Then $K$ is conjugate to $\sigma^{-1}(H)/\Gamma_3(P_3)$, where $H$ is one of the subgroups $\brak{\id}$, $\ang{(1,2)}$, $\ang{(1,2,3)}$ or $\sn[3]$ of $\sn[3]$.

\item Consider the four subgroups of $\sn[3]$ given in part~(\ref{it:4dimab0}).
\begin{enumerate}

\item\label{item:4dimab1a} If $H=\brak{\id}$, the group $\sigma^{-1}(H)/\Gamma_3(P_3)=P_3/\Gamma_3(P_3)$ has a presentation whose generators are $a=A_{1,3}$, $b=A_{2,3}$, $c=A_{1,2}$ and $d=[A_{1,2}, A_{2,3}]$, and that are subject to the following relations:
   \begin{multicols}{3}
       \begin{enumerate}[(1)]
   \item $[b,a]=d$
   \item $[c,a]=d^{-1}$
   \item $[c,b]=d$
   \item $[d,a]=1$
   \item $[d,b]=1$
   \item $[d,c]=1$.
\end{enumerate}
   \end{multicols}
\noindent
For each of the remaining groups, a presentation is obtained by adding extra generators and relations to those of $P_3/\Gamma_3(P_3)$ given in~(\ref{item:4dimab1a}). In each case, we will just indicate these extra generators and relations.

\item\label{item:4dimab1b} If $H=\ang{(1,2)}$, the group $\sigma^{-1}(H)/\Gamma_3(P_3)$ has a presentation with one extra generator $\alpha=\sigma_1$
\noindent
and five extra relations:
   \begin{multicols}{3}
       \begin{enumerate}[(1)]
    \item $\alpha^2=c$
    \item $\alpha d \alpha^{-1}=d^{-1}$
\item $\alpha a \alpha^{-1}=b$
\item $\alpha b \alpha^{-1}=ad^{-1}$
\item $\alpha c \alpha^{-1}=c$.
\end{enumerate}
   \end{multicols}
   \item If $H=\ang{(1,2,3)}$, the group $\sigma^{-1}(H)/\Gamma_3(P_3)$ has one extra generator $\alpha=\sigma_2\sigma_1^{-1}$
\noindent
and five extra relations:
      \begin{multicols}{3}
       \begin{enumerate}[(1)]
    \item $\alpha^3=d^{-1}$
    \item $\alpha d \alpha^{-1}=d$
\item $\alpha a \alpha^{-1}=bd$
\item $\alpha b \alpha^{-1}=cd^{-1}$
\item $\alpha c \alpha^{-1}=a$.
\end{enumerate}
   \end{multicols}

   \item\label{item:4dimab1d} If $H=\sn[3]$, the group $\sigma^{-1}(H)/\Gamma_3(P_3)=B_{3}/\Gamma_3(P_3)$ 
  has two extra generators $\alpha=\sigma_2\sigma_1$ and $\beta=\sigma_1$
\noindent
and eleven extra relations:
   \begin{multicols}{3}
       \begin{enumerate}[(1)]
    \item $\alpha^3=abc$
    \item $\beta^2=c$
    \item $\alpha d \alpha^{-1}=d$
    \item $\beta d \beta^{-1}=d^{-1}$
    \item $\alpha a \alpha^{-1}=b$
\item $\alpha b \alpha^{-1}=c$
\item $\alpha c \alpha^{-1}=a$
\item $\beta a \beta^{-1}=b$
\item $\beta b \beta^{-1}=ad^{-1}$
\item $\beta c \beta^{-1}=c$
\item $\beta \alpha \beta^{-1}=b^{-1}\alpha^2$.
\end{enumerate}
   \end{multicols}
  \end{enumerate}
  \end{enumerate}
  \end{thm}

\begin{proof}\mbox{}
\begin{enumerate}[(a)]
\item 
First suppose that $K_{1}$ and $K_{2}$ are subgroups of $B_3/\Gamma_3(P_3)$ that contain $P_3/\Gamma_3(P_3)$, and for which $\overline{\sigma}(K_{1})=\overline{\sigma}(K_{2})$. We claim that $K_{1}=K_{2}$. To see this, let $x\in K_{1}$. Since $\overline{\sigma}(K_{1})=\overline{\sigma}(K_{2})$, there exists $y\in K_{2}$ such that $\overline{\sigma}(x)=\overline{\sigma}(y)$, and so there exists $z\in P_{3}/\Gamma_3(P_3)$ such that $y^{-1}x=z$. But $P_{3}/\Gamma_3(P_3)\subset K_{1}\cap K_{2}$ by hypothesis, so $x=yz\in K_{2}$, which proves that $K_{1}\subset K_{2}$. A similar argument shows that $K_{2}\subset K_{1}$, which proves the claim. So if $K$ is a subgroup $B_3/\Gamma_3(P_3)$ that contains $P_3/\Gamma_3(P_3)$, $K=\overline{\sigma}^{-1}(\overline{\sigma}(K))$, in particular $K=\overline{\sigma}^{-1}(H)$, where $H=\overline{\sigma}(K)$ is a subgroup of $\sn[3]$. Since all such subgroups are normal in $\sn[3]$, with the exception of those of order $2$, to complete the proof of part~(\ref{it:4dimab0}), it suffices to show that if $K_{1}$ and $K_{2}$ are subgroups of $B_3/\Gamma_3(P_3)$ that contain $P_3/\Gamma_3(P_3)$, and for which $\overline{\sigma}(K_{1})=\ang{(1,2)}$ and $\overline{\sigma}(K_{2})=\tau$, where $\tau\in \brak{(1,3),(2,3)}$, then $K_{1}$ and $K_{2}$ are conjugate in $B_3/\Gamma_3(P_3)$. To see this, let $\tau'\in \sn[3]$ be such that $\tau' (1,2)\tau'^{-1}=\tau$, and let $\widetilde{\tau}' \in B_3/\Gamma_3(P_3)$ be such that $\overline{\sigma}(\widetilde{\tau}')=\tau'$. Then $\widetilde{\tau}' K_{1} \widetilde{\tau}'^{-1}$ contains $\ker{\overline{\sigma}}=P_3/\Gamma_3(P_3)$, and $\overline{\sigma}(\widetilde{\tau}' K_{1} \widetilde{\tau}'^{-1})=\overline{\sigma}(K_{2})=\ang{\tau}$. The result then follows from the first part of the proof.

\item  The case (i) follows from  \repr{nil}(\ref{it:relnsalphaa}).
 For  cases~(\ref{item:4dimab1b})--(\ref{item:4dimab1d}) we apply  the techniques of~\cite[Proposition~1, p.~139]{Johnson} 
to the extension   
   \begin{equation*}
    1\to P_3/\Gamma_3(P_3)\to \sigma^{-1}(H)/\Gamma_3(P_3)\to H\to 1.
   \end{equation*}
The extra relations involving the action by conjugacy of $H$ on the kernel 
follow from equations~\reqref{action1} and~\reqref{action2}.\qedhere
\end{enumerate}
\end{proof}

By \rerems{acgen}(\ref{it:acgena}), 
the groups of the form $\sigma^{-1}(H)/\Gamma_3(P_3)$ described in \reth{4dimab} are almost-crystallographic. 
Further, since $\ker{\overline{\sigma}}$ is torsion free by \relem{ptorsion}(\ref{it:ptorsiona}) and the torsion of $\sn[3]$ divides $6$, it follows from \reth{almostb} that these groups are also almost-Bieberbach. Using \reth{4dimab}, we now identify these groups 
with those given in the classification of $4$-dimensional almost-Bieberbach groups with $2$-step nilpotent 
subgroup given in~\cite[Section~7.2]{Dekimpe}. Note that by~\cite[Remark~2.5]{GPS}, if $M$ is an infra-nilmanifold whose fundamental group is $E$, then it is \emph{orientable} if and only if the image of the representation $\map{\theta_F}{F}[\operatorname{\text{GL}}(n,\Z)]$ given by \req{theta} is contained in $\operatorname{\text{SL}}(n,\Z)$.   

\begin{cor}\label{cor:4dim}
Let $H$ be a subgroup of $\sn[3]$. Then $\sigma^{-1}(H)/\Gamma_3(P_3)$ is a $4$-dimensional almost-Bieberbach group with $2$-step nilpotent 
subgroup, is the fundamental group of an orientable $4$-infra-nilmanifold $X_H$, and is isomorphic to:
\begin{enumerate}
\item\label{it:dekimpea} group number $1$, $Q=P1$, given in~\cite[p.~169]{Dekimpe} with $k_1=k_{3}=1$ and $k_2=-1$ if $H=\brak{1}$.
\item group number $9$, $Q=Cc$, given in~\cite[pp.~173-174]{Dekimpe} with $k_1=1$, $k_2=-1$, $k_3=k_{4}=0$ and non-trivial action if $H=\ang{(1,2)}$. 
\item group number $146$, $Q=R3$, given in~\cite[p.~207]{Dekimpe} with $k_1=k_2=1$ and $k_3=k_4=-1$ if $H=\ang{(1,2,3)}$.
\item\label{it:dekimped} group number $161$, $Q=R3c$, given in~\cite[p.~209]{Dekimpe} with $k_1=1$ and $k_2=k_3=k_4=k_5=0$ if $H=\sn[3]$.
\end{enumerate}
\end{cor}

\begin{rem}
If $G$ is a group and $a,b\in G$, then the notation used in~\cite{Dekimpe} for the commutator $[a,b]$, namely $[a,b]=a^{-1}b^{-1}ab$, is different to that used in this paper. However $[a,b] \equiv [a^{-1}, b^{-1}]$ modulo $\Gamma_3(G)$, so the difference in notation does not cause any problems in the identification of our subgroups with those of~\cite{Dekimpe}. 
\end{rem}

\begin{proof}[Proof of \reco{4dim}]
By \rerems{rank1}(\ref{it:rank1b}) and \repr{nilmanifold}, 
the groups $P_3/\Gamma_2(P_3)$ and $\Gamma_2(P_3)/\Gamma_3(P_3)$ are torsion free and their respective ranks are $3$ and $1$, and the Hirsch length of $P_3/\Gamma_3(P_3)$ is equal to $4$. 
So by \reco{almostb}, $\sigma^{-1}(H)/\Gamma_3(P_3)$ is a $4$-dimensional almost-Bieberbach group with $2$-step nilpotent subgroup. In order to identify the group $\sigma^{-1}(H)/\Gamma_3(P_3)$ with the corresponding group of~~\cite{Dekimpe} for each subgroup $H$ of $\sn[3]$, it suffices to use \reth{4dimab} and to apply the classification of~\cite[Section~7.2]{Dekimpe}, where the $k_{i}$ are as given in the statement of parts~(\ref{it:dekimpea})--(\ref{it:dekimped}).

  
It remains to show that for each subgroup $H$ of $\sn[3]$, the manifold $X_H$
are orientable. To see this, by the paragraph preceding the statement of \reco{4dim},
it suffices to show that the image of each representation $\map{\theta_H}{H}[\operatorname{\text{GL}}(4,\Z)]$ lies in $\operatorname{\text{SL}}(4,\Z)$.
We exhibit the matrices $\theta_H(\alpha)$ and $\theta_H(\beta)$, where
$\alpha$ and $\beta$ are the generators  given in \reth{4dimab} that act on the (ordered) elements $a,b,c$ and $d$. If $H=\brak{\id}$, the representation $\theta_H$ is clearly trivial, and the result follows.  
For the remaining cases, consider the elements $M_1=\left(\begin{smallmatrix}
0 & 1 & 0 & 0 \\
1 & 0 & 0 & 0 \\
0 & 0 & 1 & 0 \\
0 & 0 & 0 & -1 
\end{smallmatrix}\right)$ and $M_2=\left(\begin{smallmatrix}
0 & 0 & 1 & 0 \\
1 & 0 & 0 & 0 \\
0 & 1 & 0 & 0 \\
0 & 0 & 0 & 1 
\end{smallmatrix}\right)$ of $\operatorname{\text{SL}}(4,\Z)$. If $H=\ang{(1,2)}$, $\theta_H(\alpha)=M_1$, if $H=\ang{(1,2,3)}$, $\theta_H(\alpha)=M_2$, and if $H=\sn[3]$, $\theta_H(\alpha)=M_2$ and $\theta_H(\beta)=M_1$.
\end{proof}

\begin{rem}\label{rem:S4}
Since $\sn[4]$ has $11$ non-conjugate subgroups (including the trivial group and the whole group), 
then in a similar manner, we may show that there are eleven non-isomorphic almost-Bieberbach subgroups of $B_4/\Gamma_3(P_4)$ of the form ${\sigma}^{-1}(H)/\Gamma_3(P_4)$ with holonomy group $H$, each of dimension $10$ using \repr{nilmanifold},
where $H$ runs through the subgroups of $\sn[4]$.
\end{rem}

\subsection{Some  explicit   finite-order elements in $B_5/\Gamma_3(P_5)$}\label{sec:gamma5}



By~ \cite[Corollary~4]{GGO},
$B_5/\Gamma_2(P_5)$ possesses elements of order $3$ and $5$. From  \reth{almostb}, $B_5/\Gamma_3(P_5)$ does not have elements of order $3$, but \reth{ntorsion} implies that there exist elements of order $5$. 
Further, if $\gcd(n,6)=1$, elements of order $n$ in $B_n/\Gamma_3(P_n)$ may be determined explicitly using the construction given in the proof of \reth{ntorsion} provided we are able to compute $\delta_{n}^{n}$ in terms of the elements of the basis $\mathcal{B}$ of the group $\Gamma_2(P_n)/\Gamma_3(P_n)$ described in~\reqref{basis1}. We now carry out this calculation in the case $n=5$.


Using \req{conjugAij1} and with the notation of \relem{lem28}, we start by describing the action by conjugation of $\alpha_5^{-1}=\alpha_{0,5}^{-1}$, where $\alpha_{0,5}=\sigma_1\sigma_2 \sigma_3\sigma_4$, on the elements of the generating set $\brak{A_{i,j}}_{1\leq i<j\leq 5}$ of $P_5$:
\begin{equation}\label{eq:actionalpha5}
\alpha_5^{-1}=\sigma_4^{-1}\sigma_3^{-1}\sigma_2^{-1}\sigma_1^{-1} : 
\begin{cases}
A_{1,2} \mapsto  [A_{1,2}A_{1,3}A_{1,4}, A_{1,5}]A_{1,5}\\
A_{1,3} \mapsto  [A_{1,2}A_{2,3}A_{2,4}, A_{2,5}]A_{2,5}\\
A_{1,4} \mapsto  [A_{1,3}A_{2,3}A_{3,4}, A_{3,5}]A_{3,5}\\
A_{1,5} \mapsto  [A_{1,4}A_{2,4}A_{3,4}, A_{4,5}]A_{4,5}\\
A_{i,j} \mapsto A_{i-1,j-1} \;\text{if $2\leq i<j\leq 5$.}
\end{cases}
\end{equation}
From 
\req{basis1},
the following 10 elements form a basis of $\Gamma_2(P_5)/\Gamma_3(P_5)$:
\begin{equation}\label{eq:10elmt}
\begin{array}{l}
a_1=\left[A_{1,2}, A_{2,3}\right], \; a_2=\left[A_{1,2}, A_{2,5}\right], \; a_3=\left[A_{1,4}, A_{4,5}\right], \; a_4=\left[A_{3,4}, A_{4,5}\right]\\
a_5=\left[A_{2,3}, A_{3,4}\right], \;
b_1=\left[A_{1,2}, A_{2,4}\right], \; b_2=\left[A_{1,3}, A_{3,5}\right], \; b_3=\left[A_{2,4}, A_{4,5}\right]\\
b_4=\left[A_{1,3}, A_{3,4}\right], \; b_5=\left[A_{2,3}, A_{3,5}\right],
\end{array}
\end{equation}
and that using~\reqref{conjugalpha} and~\reqref{alphaijk}, 
under the action by conjugation by $\alpha_5^{-1}$ (considered as an element of $B_{5}/\Gamma_3(P_5)$), this basis splits into two orbits of length $5$ of the form:
\begin{equation*}
\text{$a_1 \mapsto  a_2  \mapsto  a_3  \mapsto  a_4 \mapsto  a_5$ and $b_1 \mapsto  b_2  \mapsto  b_3  \mapsto  b_4  \mapsto b_5$.}
\end{equation*}
In order to obtain an element $\alpha$ of finite order in $B_{5}/\Gamma_{3}(B_{5})$, by the construction of the proof of \reth{ntorsion}, it suffices to compute $\delta_5^{5}$ in terms of the basis of $\Gamma_2(P_5)/\Gamma_3(P_5)$ given in~\reqref{10elmt}. Let:
\begin{equation*}
\text{$c_1=[A_{1,2}A_{2,3}A_{2,4}, A_{2,5}]$, $c_2=[A_{1,5}A_{1,2}A_{1,3}, A_{1,4}]$ and $c_3=[A_{1,2}A_{1,3}A_{1,4}, A_{1,5}]$.}
\end{equation*}
In $\Gamma_2(P_5)/\Gamma_3(P_5)$,  we have $c_1c_2c_3=b_1^{-1}b_2^{-1}b_3^{-1}b_4^{-1}b_5^{-1}$.
To see this, recall that 
if $a,b$ and $c$ are elements of 
a group $G$, we have a Witt-Hall identity $[ab, c]=[a,[b,c]][b,c][a,c]$~\cite[Theorem~5.1]{MKS}.
So using \req{alphaijk}, in $\Gamma_2(P_5)/\Gamma_3(P_5)$ we obtain:
\begin{align}
c_1c_2c_3 & = [A_{1,2}A_{2,3}A_{2,4}, A_{2,5}][A_{1,5}A_{1,2}A_{1,3}, A_{1,4}][A_{1,2}A_{1,3}A_{1,4}, A_{1,5}]  \notag\\
& =  [A_{1,2}, A_{2,5}][A_{2,3}, A_{2,5}][A_{2,4}, A_{2,5}][A_{1,5}, A_{1,4}][A_{1,2}, A_{1,4}][A_{1,3}, A_{1,4}] \cdot  \notag\\
& \cdot [A_{1,2}, A_{1,5}][A_{1,3}, A_{1,5}][A_{1,4}, A_{1,5}] \notag\\
& =  [A_{1,2}, A_{2,5}][A_{2,3}, A_{3,5}]^{-1}[A_{2,4}, A_{4,5}]^{-1}[A_{1,4}, A_{1,5}]^{-1}[A_{1,2}, A_{2,4}]^{-1}[A_{1,3}, A_{3,4}]^{-1} \cdot  \notag\\
& \cdot [A_{1,2}, A_{2,5}]^{-1}[A_{1,3}, A_{3,5}]^{-1}[A_{1,4}, A_{1,5}] \notag\\
& = [A_{2,3}, A_{3,5}]^{-1}[A_{2,4}, A_{4,5}]^{-1}[A_{1,2}, A_{2,4}]^{-1}[A_{1,3}, A_{3,4}]^{-1}[A_{1,3}, A_{3,5}]^{-1} \notag\\
& = b_1^{-1}b_2^{-1}b_3^{-1}b_4^{-1}b_5^{-1}. \notag
\end{align}
By \req{eq15}, we have $\delta_5=A_{3,5}A_{4,5}\alpha_5^{-1}$, and using \req{actionalpha5} and \repr{nil}(\ref{it:relnsalphaa})(\ref{it:relnsalpha2}), as well as the fact that $\alpha_{5}^{-1} c_{1}\alpha_{5}= c_{2}$, in $P_5/\Gamma_3(P_5)$, we have: 
\begin{align}\label{eq:delta55}
\delta_5^5 & =  (A_{3,5}A_{4,5}\alpha_5^{-1})^5\notag\\
& =  A_{3,5}A_{4,5}(\alpha_5^{-1}A_{3,5}A_{4,5}\alpha_5) (\alpha_5^{-2}A_{3,5}A_{4,5}\alpha_5^2) (\alpha_5^{-3}A_{3,5}A_{4,5}\alpha_5^3) (\alpha_5^{-4}A_{3,5}A_{4,5}\alpha_5^4)\alpha_5^{-5}\notag\\
&= A_{3,5}A_{4,5}A_{2,4}A_{3,4}A_{1,3}A_{2,3} c_{1} A_{2,5}A_{1,2} (\alpha_{5}^{-1} c_{1} A_{2,5}\alpha_{5}) c_{3} A_{1,5}\alpha_5^{-5}\notag\\
& =  c_1c_2c_3A_{3,5}A_{4,5}A_{2,4}A_{3,4}A_{1,3}A_{2,3}A_{2,5}\underline{A_{1,2}}A_{1,4}A_{1,5}\alpha_5^{-5}\notag\\
& =  c_1c_2c_3a_2^{-1}a_1^{-1}a_1b_1^{-1}A_{1,2}A_{3,5}A_{4,5}A_{2,4}A_{3,4}\underline{A_{1,3}}A_{2,3}A_{2,5}A_{1,4}A_{1,5}\alpha_5^{-5}\notag\\
& =  c_1c_2c_3a_2^{-1}b_1^{-1}b_4^{-1}b_2^{-1}A_{1,2}A_{1,3}A_{3,5}A_{4,5}A_{2,4}A_{3,4}A_{2,3}A_{2,5}\underline{A_{1,4}}A_{1,5}\alpha_5^{-5}\notag\\
& =  c_1c_2c_3a_2^{-1}b_1^{-1}b_4^{-1}b_2^{-1}b_4b_1a_3^{-1}A_{1,2}A_{1,3}A_{1,4}A_{3,5}A_{4,5}A_{2,4}A_{3,4}A_{2,3}A_{2,5}\underline{A_{1,5}}\alpha_5^{-5}\notag\\
& =  c_1c_2c_3a_2^{-1}b_2^{-1}a_3^{-1}a_2a_3b_2A_{1,2}A_{1,3}A_{1,4}A_{1,5}A_{3,5}A_{4,5}A_{2,4}A_{3,4}\underline{A_{2,3}}A_{2,5}\alpha_5^{-5}\notag\\
& =  c_1c_2c_3a_5^{-1}a_5b_5^{-1}A_{1,2}A_{1,3}A_{1,4}A_{1,5}A_{2,3}A_{3,5}A_{4,5}\underline{A_{2,4}}A_{3,4}A_{2,5}\alpha_5^{-5}\notag\\
& =  c_1c_2c_3b_5^{-1}b_3^{-1}A_{1,2}A_{1,3}A_{1,4}A_{1,5}A_{2,3}A_{2,4}A_{3,5}A_{4,5}A_{3,4}\underline{A_{2,5}}\alpha_5^{-5}\notag\\
& =  c_1c_2c_3b_5^{-1}b_3^{-1}b_3b_5A_{1,2}A_{1,3}A_{1,4}A_{1,5}A_{2,3}A_{2,4}A_{2,5}A_{3,5}A_{4,5}\underline{A_{3,4}}\alpha_5^{-5}\notag\\
& =  c_1c_2c_3a_4^{-1}a_4A_{1,2}A_{1,3}A_{1,4}A_{1,5}A_{2,3}A_{2,4}A_{2,5}A_{3,4}A_{3,5}A_{4,5}\alpha_5^{-5}\notag\\
& =  c_1c_2c_3\alpha_5^5\alpha_5^{-5} =  b_1^{-1}b_2^{-1}b_3^{-1}b_4^{-1}b_5^{-1}.
\end{align}
The idea of the above computation is first to eliminate all of the terms involving $\alpha_{5}$ using \req{actionalpha5}, and then to move each of the underlined terms to the left one-by-one in order to create the word $A_{1,2}A_{1,3}A_{1,4}A_{1,5}A_{2,3}A_{2,4}A_{2,5}A_{3,4}A_{3,5}A_{4,5}$, which we know to be the full twist braid $\alpha_{5}^{5}$ in $P_{5}$, and so in $P_5/\Gamma_3(P_5)$. In doing so, we introduce basis elements of $\Gamma_2(P_5)/\Gamma_3(P_5)$ given by~\reqref{10elmt}, perhaps written in one of the forms of \req{alphaijk}. Note that the result is coherent with that of \repr{deltan}, and with the notation of that proposition, we have $m$. 
Setting $\theta=\prod_{i=1}^{2} \prod_{j=1}^{5}\, b_{i,j}^{r_{i,j}}$, where $b_{1,j}=a_{j}$, $b_{2,j}=b_{j}$ and $r_{i,j}\in \Z$ for all $i=1,2$ and $j=1,\ldots,5$, using the notation of \repr{deltan}(\ref{it:deltanb}), and applying the construction of \reth{ntorsion}, by~\reqref{delta55} and~\reqref{compat}, we have $m_{1,j}=0$, $m_{2,j}=-1$, and $\theta\delta_{5}$ is of order $5$ if and only if $\sum_{j=1}^{5} r_{1,j}=0$ and $\sum_{j=1}^{5} r_{2,j}=1$. So to obtain an explicit element $\theta$, it suffices to pick any integers satisfying these two relations. For example, if $r_{2,1}=1$ and $r_{2,2}=\cdots=r_{2,5}=r_{1,1}=\cdots=r_{1,5}=0$ then the element $b_1\delta_5=[A_{1,2},A_{2,4}](\sigma_4\sigma_3\sigma_2^{-1}\sigma_1^{-1})$ is of order $5$ in $B_5/\Gamma_3(P_5)$.

\begin{rem}\label{rem:b3gamma4} The 
study of the quotients $B_n/\Gamma_k(P_n)$ for $k>3$ does not appear to be an easy problem. Our approach requires a description of a basis of $\Gamma_k(P_n)/\Gamma_k(P_n)$. For example, if $n=3$ and $k=4$, a long and arduous calculation show that a basis of the group $\Gamma_3(P_3)/\Gamma_4(P_3)$, which is free Abelian of rank $2$ by~\rerem{rank1}(\ref{it:rank1b}), is given by $\brak{[[A_{1,2},A_{2,3}], A_{1,3}], [[A_{2,3},A_{1,3}], A_{1,2}]}$.
Since $\sn[3]$ has $4$ subgroups (up to isomorphism), arguing as in \rerem{S4}, we may exhibit $4$ non-isomorphic almost-Bieberbach subgroups of $B_3/\Gamma_4(P_3)$ of dimension $6$ of the form $\sigma^{-1}(H)/\Gamma_4(P_3)$ with holonomy group $H$, where $H$ is a subgroup of $\sn[3]$ and $\sigma$ is as in~\reqref{sespn}.
\end{rem}

\end{document}